\documentclass[12pt,reqno,openbib,runningheads,a4paper,portrait]{amsart}%
\usepackage{lineno}
\usepackage[authoryear,sectionbib]{natbib}
\usepackage[bookmarks=true,colorlinks=true,linkcolor=blue,citecolor=blue,a4paper]%
{hyperref}
\usepackage[labelformat=simple,labelfont={bf},textfont={up},labelsep=period]%
{caption}
\usepackage{amsfonts}
\usepackage{amsmath}
\usepackage{amssymb}
\usepackage{graphicx}%
\providecommand{\U}[1]{\protect\rule{.1in}{.1in}}
\providecommand{\U}[1]{\protect\rule{.1in}{.1in}}
\newtheorem{theorem}{Theorem}[section]

\newtheorem{lemma}{Lemma}[section]

\makeatletter
\renewcommand{\@biblabel}[1]{}
\makeatother
\begin{document}

\begin{center}
{\Large \textbf{Semiparametric tail-index estimation for randomly
right-truncated heavy-tailed data}}\medskip\medskip

{\large Saida Mancer, Abdelhakim Necir}$^{\ast},$ {\large Souad Benchaira}%
\medskip\\[0pt]

{\small \textit{Laboratory of Applied Mathematics, Mohamed Khider University,
Biskra, Algeria}}\medskip\medskip
\end{center}

\noindent\textbf{Abstract}\medskip

\noindent It was shown that when one disposes of a parametric information of
the truncation distribution, the semiparametric estimator of the distribution
function for truncated data \citep[][]{Wang89} is more efficient than the
nonparametric one. On the basis of this estimation method, we derive an
estimator for the tail index of Pareto-type distributions that are randomly
right-truncated and establish its consistency and asymptotic normality. The
finite sample behavior of the proposed estimator is carried out by simulation
study. We point out that, in terms of both bias and root of the mean squared
error, our estimator performs better than those based on nonparametric
estimation methods. An application to a real dataset of induction times of
AIDS diseases is given \ as well.\bigskip\medskip\medskip

 \medskip

\vfill

\vfill

\noindent{\small $^{\text{*}}$Corresponding author:
\texttt{necirabdelhakim@yahoo.fr} \newline\noindent\textit{E-mail
address:}\newline\texttt{mancer.saida731@gmail.com} (S.~Mancer)\newline%
\texttt{benchaira.s@hotmail.fr} (S.~Benchaira)}

\section{\textbf{Introduction\label{sec1}}}

\noindent Let $\left(  \mathbf{X}_{i},\mathbf{Y}_{i}\right)  ,$ $i=1,...,$
$N\geq1$ be a sample from a couple $\left(  \mathbf{X},\mathbf{Y}\right)  $ of
independent positive random variables (rv's) defined over a probability space
$\left(  \Omega,\mathcal{A},\mathbf{P}\right)  ,$ with continuous distribution
functions (df's) $\mathbf{F}$ and $\mathbf{G}$ respectively.$\ $Suppose that
$\mathbf{X}$ is right-truncated by $\mathbf{Y},$ in the sense that
$\mathbf{X}_{i}$ is only observed when $\mathbf{X}_{i}\leq\mathbf{Y}_{i}.$
Thus, let us denote $\left(  X_{i},Y_{i}\right)  ,$ $i=1,...,n$ to be the
observed data, as copies of a couple of dependent rv's $\left(  X,Y\right)  $
corresponding to the truncated sample $\left(  \mathbf{X}_{i},\mathbf{Y}%
_{i}\right)  ,$ $i=1,...,$ $N,$ where $n=n_{N}$ is a random sequence of
discrete rv's. By the weak law of large numbers, we have%
\begin{equation}
n/N\overset{\mathbf{P}}{\rightarrow}p:=\mathbf{P}\left(  \mathbf{X}%
\leq\mathbf{Y}\right)  =\int_{0}^{\infty}\mathbf{F}\left(  w\right)
d\mathbf{G}\left(  w\right)  ,\text{ as }N\rightarrow\infty, \label{p}%
\end{equation}
where the notation $\overset{\mathbf{P}}{\rightarrow}$ stands for the
convergence in probability. The constant $p$ corresponds to the probability of
observed sample which is supposed to be non-null, otherwise nothing is
observed. The truncation phenomena frequently occurs in medical studies, when
one wants to study the length of survival after the start of the disease: if
$\mathbf{Y}$ denotes the elapsed time between the onset of the disease and
death, and if the follow-up period starts $\mathbf{X}$ units of time after the
onset of the disease then, clearly, $\mathbf{X}$ is right-truncated by
$\mathbf{Y}.$ For concrete examples of truncated data in medical treatments
one refers, among others, to \cite{Lagakos88} and \cite{Wang89}. Truncated
data schemes may also occur in many other fields, namely actuarial sciences,
astronomy, demography and epidemiology, see for instance the textbook of
\cite{Lawless}. \medskip

\noindent From \cite{GS2015} the marginal df's $F^{\ast}$ and $G^{\ast}$
corresponding to the joint df of $\left(  X,Y\right)  $ are given by
\[
F^{\ast}\left(  x\right)  :=p^{-1}\int_{0}^{x}\overline{\mathbf{G}}\left(
w\right)  d\mathbf{F}\left(  w\right)  \text{ and }G^{\ast}\left(  x\right)
:=p^{-1}\int_{0}^{x}\mathbf{F}\left(  w\right)  d\mathbf{G}\left(  w\right)
.
\]
By the previous first equation we derive a representation of the underlying df
$\mathbf{F}$ as follows:%
\begin{equation}
\mathbf{F}\left(  x\right)  =p\int_{0}^{x}\frac{dF^{\ast}\left(  w\right)
}{\overline{\mathbf{G}}\left(  w\right)  }, \label{F-G}%
\end{equation}
which will be for a great interest thereafter. In the sequel, we are dealing
with the concept of regular variation. A function $\varphi$ is said to be
regularly varying at infinity with negative index $-1/\eta,$ notation
$\varphi\in\mathcal{RV}\left(  -1/\eta\right)  ,$ if
\begin{equation}
\varphi\left(  st\right)  /\varphi\left(  t\right)  \rightarrow s^{-1/\eta
},\text{ as }t\rightarrow\infty, \label{phi}%
\end{equation}
for $s>0.$ This convergence is known as the first-order condition of regular
variation and its corresponding uniform convergence is formulated in terms of
"Potter's inequalities" as follows: for any small $\epsilon>0,$ there exists
$t_{0}>0$ such that for any $t\geq t_{0}$ and $s\geq1$, we have%
\begin{equation}
\left(  1-\epsilon\right)  s^{-1/\eta-\epsilon}<\varphi\left(  st\right)
/\varphi\left(  t\right)  <\left(  1+\epsilon\right)  s^{-1/\eta+\epsilon}.
\label{pooter}%
\end{equation}
See for instance Proposition B.1.9 (assertion 5, page 367) in \cite{deHF06}.
The second-order condition \citep[see][]{deHS96} expresses the rate of the
convergence $\left(  \ref{phi}\right)  $ above. For any $x>0,$ we have%
\begin{equation}
\dfrac{\varphi\left(  tx\right)  /\varphi\left(  t\right)  -x^{-1/\eta}%
}{A\left(  t\right)  }\rightarrow x^{-1/\eta}\dfrac{x^{\tau/\eta}-1}{\tau\eta
},\text{ as }t\rightarrow\infty, \label{second-order}%
\end{equation}
where $\tau<0$ denotes the second-order parameter and $A$\ is a function
tending to zero and not changing signs near infinity with regularly varying
absolute value with positive index $\tau/\eta.$ A function $\varphi$ that
satisfies assumption $\left(  \ref{second-order}\right)  $ is denoted
$\varphi\in\mathcal{RV}_{2}\left(  -1/\eta;\tau,A\right)  .$ We now have
enough material to tackle the main goal of the paper. To begin, let us assume
that the tails of both df's $\mathbf{F}$ and $\mathbf{G}$ are regularly
varying. That is%
\begin{equation}
\overline{\mathbf{F}}\in\mathcal{RV}\left(  -1/\gamma_{1}\right)  \text{ and
}\overline{\mathbf{G}}\in\mathcal{RV}\left(  -1/\gamma_{2}\right)  ,\text{
with }\gamma_{1},\gamma_{2}>0. \label{rv1}%
\end{equation}
Under this assumption, \cite{GS2015} showed that%
\begin{equation}
\overline{F}^{\ast}\in\mathcal{RV}\left(  -1/\gamma_{1}\right)  \text{ and
}\overline{G}^{\ast}\in\mathcal{RV}\left(  -1/\gamma_{2}\right)  ,
\label{F-G-stars}%
\end{equation}
where%
\begin{equation}
\gamma:=\frac{\gamma_{1}\gamma_{2}}{\gamma_{1}+\gamma_{2}}.
\label{gamma-ratio}%
\end{equation}
For details on the proof of this statement, on refers to \cite{BchMN-16a}
(Lemma A1). The estimation of the tail index $\gamma_{1}$ was recently
addressed for the first time in \cite{GS2015} where the authors used equation
$\left(  \ref{gamma-ratio}\right)  $ to propose an estimator to $\gamma_{1}$
as a ratio of Hill estimators \citep{Hill75} of the tail indices $\gamma$ and
$\gamma_{2}.$ These estimators are based on the top order statistics
$X_{n-k:n}\leq...\leq X_{n:n}$ and $Y_{n-k:n}\leq...\leq Y_{n:n}$ pertaining
to the samples $\left(  X_{1},...,X_{n}\right)  $ and $\left(  Y_{1}%
,...,Y_{n}\right)  $ respectively. The sample fraction $k=k_{n}$ being a
sequence of integers such that, $k_{n}\rightarrow\infty$\textbf{ }and\textbf{
}$k_{n}/n\rightarrow0$\textbf{ }as\textbf{ }$n\rightarrow\infty.$ The
asymptotic normality of the given estimator is established in \cite{BchMN-15}
by considering both the tail dependence and the second-order conditions of
regular variation. By using a Lynden-bell integral, \cite{WW2016} proposed the
following estimator for the tail index $\gamma_{1}:$%
\[
\widehat{\gamma}_{1}^{\left(  \mathbf{W}\right)  }\left(  u\right)  :=\frac
{1}{\overline{\mathbf{F}}_{n}^{\left(  \mathbf{1}\right)  }\left(  u\right)
}\sum_{i=1}^{n}\mathbf{1}\left(  X_{i}>u\right)  \frac{\mathbf{F}_{n}^{\left(
\mathbf{1}\right)  }\left(  X_{i}\right)  }{C_{n}\left(  X_{i}\right)  }%
\log\frac{X_{i}}{u},
\]
where $u>0$ is a given deterministic threshold and%
\[
\mathbf{F}_{n}^{\left(  \mathbf{1}\right)  }\left(  x\right)  :=%
%TCIMACRO{\dprod \limits_{X_{i}>x}}%
%BeginExpansion
{\displaystyle\prod\limits_{X_{i}>x}}
%EndExpansion
\left[  1-\frac{1}{nC_{n}\left(  X_{i}\right)  }\right]  ,
\]
and%
\[
C_{n}\left(  x\right)  :=\frac{1}{n}\sum\limits_{i=1}^{n}\mathbf{1}\left(
X_{i}\leq x\leq Y_{i}\right)  ,
\]
is the well-known nonparametric maximum likelihood estimator introduced in the
well-known work \cite{Lynden71}. Independently, \cite{BchMN-16a} used a
Woodroofe-integral with a random threshold, to derive the following estimator
\begin{equation}
\widehat{\gamma}_{1}^{\left(  \mathbf{BMN}\right)  }:=\frac{1}{\overline
{\mathbf{F}}_{n}^{\left(  \mathbf{2}\right)  }\left(  X_{n-k:n}\right)  }%
\sum_{i=1}^{k}\frac{\mathbf{F}_{n}^{\left(  \mathbf{2}\right)  }\left(
X_{n-i+1:n}\right)  }{C_{n}\left(  X_{n-i+1:n}\right)  }\log\frac{X_{n-i+1:n}%
}{X_{n-k:n}}, \label{BMN}%
\end{equation}
where%
\[
\mathbf{F}_{n}^{\left(  \mathbf{2}\right)  }\left(  x\right)  :=\prod
_{X_{i}>x}\exp\left\{  -\dfrac{1}{nC_{n}\left(  X_{i}\right)  }\right\}  ,
\]
is the so-called Woodroofe's nonparametric estimator \citep{W-85} of df
$\mathbf{F.}$ To improve the performance of $\widehat{\gamma}_{1}^{\left(
\mathbf{BML}\right)  },$ \cite{BchMN-16b} and \cite{Haouas19} respectively
proposed a Kernel-smoothed and a reduced-biais versions of this estimator and
establish their consistency and asymptotic normality. It is worth mentioning
that the Lynden-Bell integral estimator $\widehat{\gamma}_{1}^{\left(
\mathbf{W}\right)  }$ with a random threshold $u=X_{n-k:n}$ becomes\
\begin{equation}
\widehat{\gamma}_{1}^{\left(  \mathbf{W}\right)  }:=\frac{1}{\overline
{\mathbf{F}}_{n}^{\left(  \mathbf{1}\right)  }\left(  X_{n-k:n}\right)  }%
\sum_{i=1}^{k}\frac{\mathbf{F}_{n}^{\left(  \mathbf{1}\right)  }\left(
X_{n-i+1:n}\right)  }{C_{n}\left(  X_{n-i+1:n}\right)  }\log\frac{X_{n-i+1:n}%
}{X_{n-k:n}}. \label{WW}%
\end{equation}
In a simulation study, \cite{Haouas18} compared this estimator with
$\widehat{\gamma}_{1}^{\left(  \mathbf{BMN}\right)  }.$ They pointed out that
both estimators have similar behaviors in terms of biases a \ nd mean squared
errors.\medskip

\noindent Recall that the nonparametric Lynden-Bell estimator $\mathbf{F}%
_{n}^{\left(  \mathbf{1}\right)  }$ was constructed on the basis of the fact
that $\mathbf{F}$ and $\mathbf{G}$ are both unknown. In this paper, we are
dealing with the situation when $\mathbf{F}$ is unknown but $\mathbf{G}$ is
parametrized by a known model $\mathbf{G}_{\theta},$ $\theta\in\Theta
\subset\mathbb{R}^{d},$ $d\geq1$ having a density $\mathbf{g}_{\theta}$ with
respect to Lebesgue measure. \cite{Wang89} considered this assumption and
introduced a semiparametric estimator for df $\mathbf{F}$ defined by%
\begin{equation}
\mathbf{F}_{n}\left(  x;\widehat{\theta}_{n}\right)  :=P_{n}\left(
\widehat{\theta}\right)  \frac{1}{n}\sum_{i=1}^{n}\frac{\mathbf{1}\left(
X_{i}\leq x\right)  }{\overline{\mathbf{G}}_{\widehat{\theta}}\left(
X_{i}\right)  }, \label{SCMLE}%
\end{equation}
where $1/P_{n}\left(  \widehat{\theta}\right)  :=n^{-1}\sum_{i=1}%
^{n}1/\overline{\mathbf{G}}_{\widehat{\theta}}\left(  X_{i}\right)  $ and
\begin{equation}
\widehat{\theta}:=\arg\max_{\theta\in\Theta}%
%TCIMACRO{\dprod _{i=1}^{n}}%
%BeginExpansion
{\displaystyle\prod_{i=1}^{n}}
%EndExpansion
g_{\theta}\left(  Y_{i}\right)  /\overline{\mathbf{G}}_{\theta}\left(
X_{i}\right)  , \label{CLME}%
\end{equation}
denoting the conditional maximum likelihood estimator (CMLE) of $\theta
,$\ which is consistent and asymptotically normal, see for instance
\cite{Anders70}. On the other hand, \cite{Wang89} showed that $\mathbf{F}%
_{n}\left(  x;\widehat{\theta}_{n}\right)  $\ is a uniformly consistent
estimator over the $x$-axis and established, under suitable regularity
assumptions, its asymptotic normality. Both \cite{Wang89} and
\cite{Moreira2010} pointed out that the semiparametric estimate has greater
efficiency uniformly over the $x$-axis. In the light of a simulation study,
the authors suggest that the semiparametric estimate is a better choice when
parametric information of the truncation distribution is available. Since the
apparition of this estimation method many papers are devoted to the
statistical inference with truncation data, see for instance \cite{BW1996},
\cite{Li97}, \cite{Qin2001}, \cite{shen2010}, \cite{Moreira14}, and
\cite{sh2020}.\medskip

\noindent Motivated by the features of the semiparametric estimation, we next
propose a new an estimator for $\gamma_{1}$ by means of a suitable functional
of $\mathbf{F}_{n}\left(  x;\widehat{\theta}_{n}\right)  .$ We start our
construction by noting that from Theorem 1.2.2 in de \cite{deHF06}, the
first-order condition $\left(  \ref{rv1}\right)  $ (for $\mathbf{F}$) implies
that%
\begin{equation}
\lim_{t\rightarrow\infty}\frac{1}{\overline{\mathbf{F}}\left(  t\right)  }%
\int_{t}^{\infty}\log\left(  x/t\right)  d\mathbf{F}\left(  x\right)
=\gamma_{1}. \label{gamma1}%
\end{equation}
In other words, $\gamma_{1}$ may viewed as a functional $\psi_{t}\left(
\mathbf{F}\right)  ,$ for a large $t,$ where
\[
\psi_{t}\left(  \mathbf{F}\right)  :=\frac{1}{\overline{\mathbf{F}}\left(
t\right)  }\int_{t}^{\infty}\log\left(  x/t\right)  d\mathbf{F}\left(
x\right)  .
\]
Replacing $\mathbf{F}$ by $\mathbf{F}_{n}\left(  \cdot;\widehat{\theta}%
_{n}\right)  $ and letting $t=X_{n-k:n}$ yield%
\begin{align}
\widehat{\gamma}_{1}  &  =\psi_{X_{n-k:n}}\left(  \mathbf{F}_{n}\left(
\cdot;\widehat{\theta}_{n}\right)  \right) \label{gama1f}\\
&  =\frac{1}{\overline{\mathbf{F}}_{n}\left(  X_{n-k:n};\widehat{\theta}%
_{n}\right)  }\int_{X_{n-k:n}}^{\infty}\log\left(  x/X_{n-k:n}\right)
d\mathbf{F}_{n}\left(  x;\widehat{\theta}_{n}\right)  ,
\end{align}
as new estimator for $\gamma_{1}.$ Observe that
\begin{align*}
&  \int_{t}^{\infty}\log\left(  x/t\right)  d\mathbf{F}_{n}\left(
x;\widehat{\theta}_{n}\right) \\
&  =P_{n}\left(  \widehat{\theta}\right)  \int_{X_{n-k:n}}^{\infty}\log\left(
x/X_{n-k:n}\right)  \mathbf{1}\left(  x\geq X_{n-k}\right)  d\mathbf{F}%
_{n}\left(  x;\widehat{\theta}_{n}\right)  ,
\end{align*}
which may be rewritten into%
\begin{align*}
&  \frac{P_{n}\left(  \widehat{\theta}\right)  1}{n}\sum_{i=1}^{n}%
\int_{X_{n-k:n}}^{\infty}\frac{\log\left(  x/X_{n-k:n}\right)  \mathbf{1}%
\left(  x\geq X_{n-k}\right)  }{\overline{\mathbf{G}}_{\widehat{\theta}%
}\left(  X_{i}\right)  }d\mathbf{1}\left(  X_{i}\leq x\right) \\
&  =P_{n}\left(  \widehat{\theta}\right)  \frac{1}{n}\sum_{i=1}^{k}\frac
{\log\left(  X_{n-i+1}/X_{n-k:n}\right)  }{\overline{\mathbf{G}}%
_{\widehat{\theta}}\left(  X_{n-i+1:n}\right)  }.
\end{align*}
On the other hand, $\mathbf{F}\left(  X_{n-k:n};\widehat{\theta}_{n}\right)  $
equals
\[
P_{n}\left(  \widehat{\theta}\right)  \frac{1}{n}\sum_{i=1}^{n}\frac
{\mathbf{1}\left(  X_{i:n}\leq X_{n-k:n}\right)  }{\overline{\mathbf{G}%
}_{\widehat{\theta}}\left(  X_{i:n}\right)  }=P_{n}\left(  \widehat{\theta
}\right)  \frac{1}{n}\sum_{i=1}^{n-k}1/\overline{\mathbf{G}}_{\widehat{\theta
}}\left(  X_{i:n}\right)  .
\]
Hence%
\begin{align*}
\overline{\mathbf{F}}\left(  X_{n-k:n};\widehat{\theta}_{n}\right)   &
=\frac{\dfrac{1}{n}%
%TCIMACRO{\dsum _{i=1}^{n}}%
%BeginExpansion
{\displaystyle\sum_{i=1}^{n}}
%EndExpansion
1/\overline{\mathbf{G}}_{\widehat{\theta}}\left(  X_{i:n}\right)  -\dfrac
{1}{n}%
%TCIMACRO{\dsum _{i=1}^{n-k}}%
%BeginExpansion
{\displaystyle\sum_{i=1}^{n-k}}
%EndExpansion
1/\overline{\mathbf{G}}_{\widehat{\theta}}\left(  X_{i:n}\right)  }{\dfrac
{1}{n}%
%TCIMACRO{\dsum _{i=1}^{n}}%
%BeginExpansion
{\displaystyle\sum_{i=1}^{n}}
%EndExpansion
1/\overline{\mathbf{G}}_{\widehat{\theta}}\left(  X_{i:n}\right)  }\\
&  =P_{n}\left(  \widehat{\theta}\right)  \frac{1}{n}\sum_{i=1}^{k}%
1/\overline{\mathbf{G}}_{\widehat{\theta}}\left(  X_{n-i+1:n}\right)  .
\end{align*}
Thereby, the final form of our new estimator is
\begin{equation}
\widehat{\gamma}_{1}=\frac{\sum_{i=1}^{k}\left(  \overline{\mathbf{G}%
}_{\widehat{\theta}}\left(  X_{n-i+1:n}\right)  \right)  ^{-1}\log\left(
X_{n-i+1}/X_{n-k:n}\right)  }{\sum_{i=1}^{k}\left(  \overline{\mathbf{G}%
}_{\widehat{\theta}}\left(  X_{n-i+1:n}\right)  \right)  ^{-1}}.
\label{estimator}%
\end{equation}
The asymptotic behavior of $\widehat{\gamma}_{1}$ will be established by means
of the following tail empirical process%
\[
\mathbf{D}_{n}\left(  x;\widehat{\theta};\gamma_{1}\right)  :=\sqrt{k}\left(
\frac{\overline{\mathbf{F}}_{n}\left(  xX_{n-k:n};\widehat{\theta}\right)
}{\overline{\mathbf{F}}_{n}\left(  X_{n-k:n};\widehat{\theta}\right)
}-x^{-1/\gamma_{1}}\right)  ,\text{\textbf{\ }for\textbf{\ }}x>1.
\]
This method was already used to establish the asymptotic behavior of Hill's
estimator for complete data (\citeauthor{deHF06}, \citeyear{deHF06}, page 162)
that we will adapt to the truncation case. Indeed, an integration by parts of
the integral $\left(  \ref{gama1f}\right)  ,$ yields%
\[
\widehat{\gamma}_{1}=\int_{1}^{\infty}x^{-1}\frac{\overline{\mathbf{F}}%
_{n}\left(  xX_{n-k:n};\widehat{\theta}\right)  }{\overline{\mathbf{F}}%
_{n}\left(  X_{n-k:n};\widehat{\theta}\right)  }dx,
\]
and therefore%
\begin{equation}
\sqrt{k}\left(  \widehat{\gamma}_{1}-\gamma_{1}\right)  =\int_{1}^{\infty
}x^{-1}\mathbf{D}_{n}\left(  x;\widehat{\theta};\gamma_{1}\right)  dx.
\label{rep}%
\end{equation}
Thus for a suitable weighted weak approximation to $\mathbf{D}_{n}\left(
\cdot;\widehat{\theta}\right)  ,$\ we may easily deduce the consistency and
asymptotic normality of $\widehat{\gamma}_{1}.$\ This process may also
contribute to the goodness-of-fit test to fitting heavy-tailed distributions
via, among others, the Kolmogorov-Smirnov and Cramer-Von Mises type statistics%
\[
\sup_{x>1}\left\vert \mathbf{D}_{n}\left(  x;\widehat{\theta},\widehat{\gamma
}_{1}\right)  \right\vert \text{ and }\int_{1}^{\infty}\mathbf{D}_{n}%
^{2}\left(  x;\widehat{\theta},\widehat{\gamma}_{1}\right)  dx^{-1/\widehat
{\gamma}_{1}}.
\]
\ More precisely, these statistics are used when testing the null hypothesis
$H_{0}:$"both $\mathbf{F}$\ and $\mathbf{G}$\ are heavy-tailed" versus the
alternative one $H_{1}:$\ "at least one of $\mathbf{F}$\ and $\mathbf{G}\ $is
not heavy-tailed", that is $H_{0}:$"$\left(  \ref{rv1}\right)  $\ holds"
versus $H_{1}:$\ "$\left(  \ref{rv1}\right)  $\ does not hold". This problem
has been already addressed by \cite{DH2006} and \cite{KP8} in the case of
complete data. The (uniform) weighted weak convergence of $\mathbf{D}%
_{n}\left(  x;\widehat{\theta},\gamma_{1}\right)  $\ and the asymptotic
normality of $\widehat{\gamma}_{1},$\ stated below, will be of great interest
to establish the limit distributions of the aforementioned test statistics.
This is out of the scope of this paper whose remainder is structured as
follows. In Section \ref{sec2}, we present our main results which consist in
the consistency and asymptotic normality of estimator $\widehat{\gamma}_{1}.$
The performance of the proposed estimator is checked by simulation in Section
\ref{sec3}. An application to a real dataset composed of induction times of
AIDS diseases is given in Section \ref{sec4}. All proofs are gathered in
Section \ref{sec5}. The proofs of two useful lemmas are postponed to the Appendix.

\section{\textbf{Main results\label{sec2}}}

\noindent The regularity assumptions, denoted $\left[  A0\right]  ,$
concerning the existence, consistency and asymptotic normality of the CLME
estimator $\widehat{\theta},$ given in $\left(  \ref{CLME}\right)  ,$ are
discussed in \cite{Anders70}. Here we only state additional conditions on df
$\mathbf{G}_{\theta}$ corresponding to Pareto-type models which are required
to establish the asymptotic behavior of our newly estimator $\widehat{\gamma
}_{1}.$

\begin{itemize}
\item $\left[  A1\right]  $ For each fixed $y,$ the function $\theta
\rightarrow\mathbf{G}_{\theta}\left(  y\right)  $ is continuously
differentiable of partial derivatives $\mathbf{G}_{\theta}^{\left(  j\right)
}=:\partial\mathbf{G}_{\theta}/\partial\theta_{j},$ $j=1,...,d.$

\item $\left[  A2\right]  $ $\overline{\mathbf{G}}_{\theta}^{\left(  j\right)
}\in\mathcal{RV}\left(  -1/\gamma_{2}\right)  .$

\item $\left[  A3\right]  $ $y^{-\epsilon}\overline{\mathbf{G}}_{\theta
}^{\left(  j\right)  }\left(  y\right)  /\overline{\mathbf{G}}_{\theta}\left(
y\right)  \rightarrow0,$ as $y\rightarrow\infty,$ for any $\epsilon
>0.\medskip$
\end{itemize}

\noindent For commonly Pareto-type models, one may easily checked that there
exist some constants $a_{j}\geq0,$ $c_{j}$ and $d_{j},$ such that
$\overline{\mathbf{G}}_{\theta}^{\left(  j\right)  }\left(  y\right)  \sim
c_{j}\left(  y^{-1/\gamma_{2}}+d_{j}\right)  \log y,$ for all large $x.$ Then
one may consider that the assumptions $\left[  A1\right]  -\left[  A3\right]
$ are not very restrictive and they may be acceptable in the extreme value theory.

\begin{theorem}
\label{Theorem1}Assume that $\overline{\mathbf{F}}\in\mathcal{RV}_{2}\left(
-1/\gamma_{1};\rho_{1},\mathbf{A}\right)  $\textbf{\ }and $\mathbf{G}_{\theta
}\in\mathcal{RV}\left(  -1/\gamma_{2}\right)  $ satisfying the assumptions
$\left[  A0\right]  -\left[  A3\right]  ,$ and suppose that $\gamma_{1}%
<\gamma_{2}.$ Then on the probability space $\left(  \Omega,\mathcal{A}%
,\mathbf{P}\right)  ,$ there exists a standard Wiener process $\left\{
W\left(  s\right)  ,0\leq s\leq1\right\}  $ such that, for any small
$0<\epsilon<1/2:$
\[
\sup_{x>1}x^{\epsilon}\left\vert \mathbf{D}_{n}\left(  x;\widehat{\theta
},\gamma_{1}\right)  -\Gamma\left(  x;W\right)  -x^{-1/\gamma_{1}}%
\dfrac{x^{\rho_{1}/\gamma_{1}}-1}{\rho_{1}\gamma_{1}}\sqrt{k}\mathbf{A}\left(
a_{k}\right)  \right\vert \overset{\mathbf{P}}{\rightarrow}0,
\]
provided that $\sqrt{k}\mathbf{A}\left(  a_{k}\right)  =O\left(  1\right)  ,$
where%
\begin{align*}
&  \Gamma\left(  x;W\right)
\begin{array}
[c]{c}%
:=
\end{array}
\frac{\gamma}{\gamma_{1}}x^{-1/\gamma_{1}}\left\{  x^{1/\gamma}W\left(
x^{-1/\gamma}\right)  -W\left(  1\right)  \right\} \\
&  \ \ \ \ \ \ \ \ \ \ \ \ \ \ \ \ \ \ \ \ \ \ +\frac{\gamma}{\gamma
_{1}+\gamma_{2}}x^{-1/\gamma_{1}}\int_{0}^{1}s^{-\gamma/\gamma_{2}-1}\left\{
x^{1/\gamma}W\left(  x^{-1/\gamma}s\right)  -W\left(  s\right)  \right\}  ds,
\end{align*}
is a centred Gaussian process and $a_{k}:=F^{\ast\leftarrow}\left(
1-k/n\right)  ,$ where
\[
F^{\ast\leftarrow}\left(  s\right)  :=\inf\left\{  x:F^{\ast}\left(  x\right)
\geq s\right\}  ,\text{ }0<s<1,
\]
denotes the quantile (or the generalized inverse) function pertaining to df
$F^{\ast}.$
\end{theorem}

\noindent By a strength application of this weak approximation, we establish
both consistency and asymptotic normality of our newly estimator
$\widehat{\gamma}_{1},$ that we state there in the following Theorem.

\begin{theorem}
\label{Theorem2}Under the assumptions of Theorem $\ref{Theorem1},$ we have%
\begin{align*}
&  \widehat{\gamma}_{1}-\gamma_{1}\\
&  =k^{-1/2}\int_{1}^{\infty}x^{-1}\Gamma\left(  x;W\right)  dx+\mathbf{A}%
\left(  a_{k}\right)  \int_{1}^{\infty}x^{-1/\gamma_{1}-1}\dfrac{x^{\rho
_{1}/\gamma_{1}}-1}{\rho_{1}\gamma_{1}}dx+o_{\mathbf{P}}\left(  k^{-1/2}%
\right)  ,
\end{align*}
this implies that $\widehat{\gamma}_{1}\overset{\mathbf{P}}{\rightarrow}%
\gamma_{1}.$ Whenever $\sqrt{k}\mathbf{A}\left(  a_{k}\right)  \rightarrow
\lambda<\infty,$ we get
\[
\sqrt{k}\left(  \widehat{\gamma}_{1}-\gamma_{1}\right)  \overset{\mathcal{D}%
}{\rightarrow}\mathcal{N}\left(  \frac{\lambda}{1-\rho_{1}},\sigma^{2}\right)
,
\]
where $\sigma^{2}:=\gamma^{2}\left(  1+\gamma_{1}/\gamma_{2}\right)  \left(
1+\left(  \gamma_{1}/\gamma_{2}\right)  ^{2}\right)  \left(  1-\gamma
_{1}/\gamma_{2}\right)  ^{3}\mathbf{1}\left(  \gamma_{1}<\gamma_{2}\right)  ,$
and $\mathbf{1}\left(  \mathcal{A}\right)  $ stands for the indicator function
pertaining to a set $\mathcal{A}.$
\end{theorem}

\section{\textbf{Simulation study \label{sec3}}}

\noindent In this section we will perform a simulation study in order to
compare the finite sample behavior of our the newly semiparametric estimator
$\widehat{\gamma}_{1}$, \ given in $\left(  \ref{estimator}\right)  ,$ with
the Woodrofee and the Lynden-Bell integral estimators $\widehat{\gamma}%
_{1}^{\left(  \mathbf{BMN}\right)  }$ and $\widehat{\gamma}_{1}^{\left(
\mathbf{W}\right)  },$ given respectively in $\left(  \ref{BMN}\right)  $ and
$\left(  \ref{WW}\right)  .$ The truncation and truncated distributions
functions $\mathbf{F}$ and $\mathbf{G},$ will be chosen among the following
two models:

\begin{itemize}
\item Burr $\left(  \gamma,\delta\right)  $ distribution with right-tail
function:%
\[
\overline{H}\left(  x\right)  =\left(  1+x^{1/\delta}\right)  ^{-\delta
/\gamma},\text{ }x\geq0,\text{ }\delta>0,\text{ }\gamma>0;
\]

\item Fr\'{e}chet $\left(  \gamma\right)  $ distribution with right-tail
function:%
\[
\overline{H}\left(  x\right)  =1-\exp\left(  -x^{-1/\gamma}\right)
,x>0,\gamma>0.
\]

\end{itemize}

\noindent The simulation study be made in fours scenarios following to the
choice of the underlying df's $\mathbf{F}$ and $\mathbf{G}_{\theta}\mathbf{:}%
$\smallskip

\begin{itemize}
\item $\left[  S1\right]  $ Burr $\left(  \gamma_{1},\delta\right)  $
truncated by Burr $\left(  \gamma_{2},\delta\right)  ;$ with $\theta=\left(
\gamma_{2},\delta\right)  $

\item $\left[  S2\right]  $ Fr\'{e}chet $\left(  \gamma_{1}\right)  $
truncated by Fr\'{e}chet $\left(  \gamma_{2}\right)  ;$ with $\theta
=\gamma_{2}$

\item $\left[  S3\right]  $ Fr\'{e}chet $\left(  \gamma_{1}\right)  $
truncated by Burr $\left(  \gamma_{2},\delta\right)  ;$ with $\theta=\left(
\gamma_{2},\delta\right)  $

\item $\left[  S4\right]  $ Burr $\left(  \gamma_{1},\delta\right)  $
truncated by Fr\'{e}chet $\left(  \gamma_{2}\right)  ;$ with $\theta
=\gamma_{2}$
\end{itemize}

\noindent To this end, we fix $\delta=1/4$ and choose the values $0.6$ and
$0.8$ for $\gamma_{1}$ and $55\%$ and $90\%$ for the portions of observed
truncated data given in $\left(  \ref{p}\right)  $ by%
\begin{equation}
p=\int_{0}^{\infty}\mathbf{F}\left(  w\right)  d\mathbf{G}_{\theta}\left(
w\right)  , \label{e}%
\end{equation}
so that the assumption $\gamma_{1}<$ $\gamma_{2}$ stated in Theorem
$\ref{Theorem1}$ be hold. In other terms the values of $p$ have to be greater
than $50\%.$ For each couple $\left(  \gamma_{1},p\right)  ,$ we solve the
equation $\left(  \ref{e}\right)  $ to get the pertaining $\gamma_{2}$-value,
which we summarize as follows:
\begin{equation}
\left(  p,\gamma_{1},\gamma_{2}\right)  =\left(  55\%,0.6,1.4\right)  ,\left(
90\%,0.6,5.4\right)  ,\left(  55\%,0.8,1.9\right)  ,\left(
90\%,0.8,7.2\right)  . \label{values}%
\end{equation}
For each scenario, we simulated $1000$ random samples of size $N$ $=300$ and
compute the root mean squared error (RMSE) and the absolute bias (ABIAS)
corresponding to each estimator $\widehat{\gamma}_{1},$ $\widehat{\gamma}%
_{1}^{\left(  \mathbf{BMN}\right)  }$ and $\widehat{\gamma}_{1}^{\left(
\mathbf{W}\right)  }.$ The comparison is carry out by plotting the ABIAS and
RMSE as functions of the sample fraction $k$ which is vary from $2$ to $120.$
The end points of this range is chosen so that it contains the optimal number
of upper extremes $k^{\ast}$ used in the computation of the tail index
estimate. There are many heuristic methods to select the optimal choice of
$k^{\ast},$ see for instance \cite{CG-15}, here we use the algorithm proposed
by \cite{ReTo7} in page 137, which is incorporated in the R software
\textquotedblleft Xtremes\textquotedblright package. Note that the computation
the CLME of $\theta$ is made by means of the syntax "maxLik" of the maxLik R
software package. The optimal sample fraction $k^{\ast}$ is defined, in this
procedure, by%
\[
k^{\ast}:=\arg\min_{1<k<n}\frac{1}{k}\sum_{i=1}^{k}i^{\theta}\left\vert
\widehat{\gamma}\left(  i\right)  -\text{median}\left\{  \widehat{\gamma
}\left(  1\right)  ,...,\widehat{\gamma}\left(  k\right)  \right\}
\right\vert ,
\]
for suitable constant $0\leq\theta\leq1/2,$ where $\widehat{\gamma}\left(
i\right)  $ corresponds to an estimator of tail index $\gamma,$ based on the
$i$ upper order statistics, of a Pareto-type model. We observed, in our
simulation study, that $\theta=0.3$ allows better results both in terms of
bias and rmse. It is worth mentioning that making $N$ vary did not provide
notable findings, therefore we kept the size $N$ be fixed. The finite sample
behavior of the above mentioned estimators are illustrated in Figures
\ref{s1}-\ref{s8}. On the overall, the biases of three estimators are almost
equal, however in the case of moderate truncation $\left(  p\approx
50\%\right)  $ the RMSE of our newly semiparametric $\widehat{\gamma}_{1}$ is
clearly the smaller compared that of $\widehat{\gamma}_{1}^{\left(
\mathbf{BMN}\right)  }$ and $\widehat{\gamma}_{1}^{\left(  \mathbf{W}\right)
}.$ Actually, the moderate truncation situation is the most frequently in real
data, while up to our knowledge the strong truncation remains theoretic. In
this sense, we may consider that the semiparametric estimator is more
efficient than the two other ones. We point out that the two estimators
$\widehat{\gamma}_{1}^{\left(  \mathbf{BMN}\right)  }$ and $\widehat{\gamma
}_{1}^{\left(  \mathbf{W}\right)  }$ have almost the same behavior which
actually is noticed before by \cite{Haouas18}. The optimal sample fractions
$k^{\ast}$ of each tail index estimator are given in Tables \ref{tab1}%
-\ref{tab4}.

\begin{center}%
%TCIMACRO{\FRAME{ftbpFU}{4.0413in}{3.9185in}{0pt}{\Qcb{Absolute bias (left two
%panels) and RMSE (right two panels) of $\widehat{\gamma}_{1}$ (black) and
%$\widehat{\gamma}_{1}^{\left(  \QTR{bf}{BMN}\right)  }$ (red) and
%$\widehat{\gamma}_{1}^{\left(  \QTR{bf}{W}\right)  }$(blue), corresponding to
%two situations of scenario $S_{1}:\left(  \gamma_{1}=0.6,p=55\%\right)  $ and
%$\left(  \gamma_{1}=0.6,p=90\%\right)  $ based on $1000$ samples of size
%$300.$}}{\Qlb{s1}}{sinario1.eps}{\special{ language "Scientific Word";
%type "GRAPHIC";  maintain-aspect-ratio TRUE;  display "USEDEF";
%valid_file "F";  width 4.0413in;  height 3.9185in;  depth 0pt;
%original-width 5.9326in;  original-height 5.7493in;  cropleft "0";
%filename 'Sinario1.eps';file-properties "XNPEU";}}}%
%BeginExpansion
\begin{figure}
[ptb]
\begin{center}
\includegraphics[
height=3.9185in,
width=4.0413in
]%
{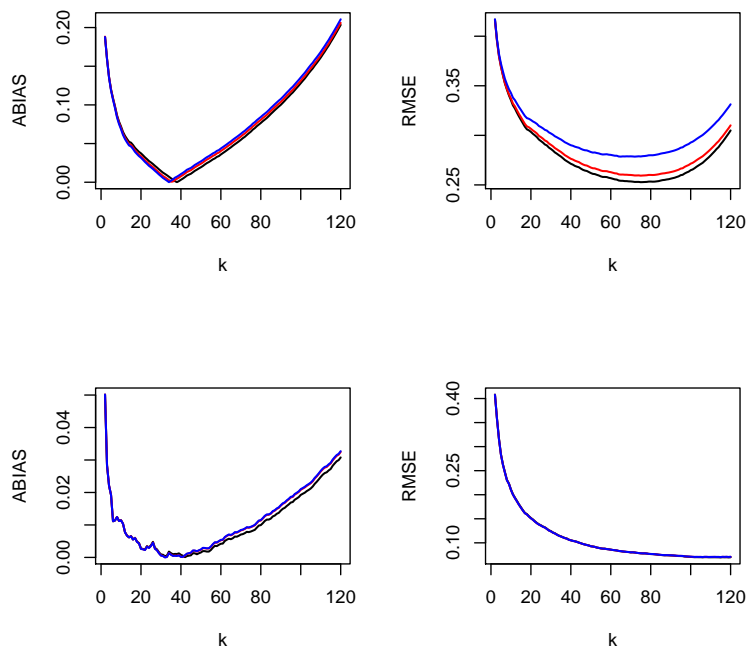}%
\caption{Absolute bias (left two panels) and RMSE (right two panels) of
$\widehat{\gamma}_{1}$ (black) and $\widehat{\gamma}_{1}^{\left(
\mathbf{BMN}\right)  }$ (red) and $\widehat{\gamma}_{1}^{\left(
\mathbf{W}\right)  }$(blue), corresponding to two situations of scenario
$S_{1}:\left(  \gamma_{1}=0.6,p=55\%\right)  $ and $\left(  \gamma
_{1}=0.6,p=90\%\right)  $ based on $1000$ samples of size $300.$}%
\label{s1}%
\end{center}
\end{figure}
%EndExpansion
%

%TCIMACRO{\FRAME{ftbpFU}{4.0439in}{3.915in}{0pt}{\Qcb{Absolute bias (left two
%panels) and RMSE (right two panels) of $\widehat{\gamma}_{1}$ (black) and
%$\widehat{\gamma}_{1}^{\left(  \QTR{bf}{BMN}\right)  }$ (red) and
%$\widehat{\gamma}_{1}^{\left(  \QTR{bf}{W}\right)  }$(blue), corresponding to
%two situations of scenario $S_{1}:\left(  \gamma_{1}=0.8,p=55\%\right)  $ and
%$\left(  \gamma_{1}=0.8,p=90\%\right)  $ based on $1000$ samples of size
%$300.$}}{\Qlb{s2}}{sinario1(g1=0.8).eps}%
%{\special{ language "Scientific Word";  type "GRAPHIC";  display "USEDEF";
%valid_file "F";  width 4.0439in;  height 3.915in;  depth 0pt;
%original-width 5.7648in;  original-height 5.7493in;  cropleft "0";
%filename 'Sinario1(g1=0.8).eps';file-properties "XNPEU";}}}%
%BeginExpansion
\begin{figure}
[ptb]
\begin{center}
\includegraphics[
height=3.915in,
width=4.0439in
]%
{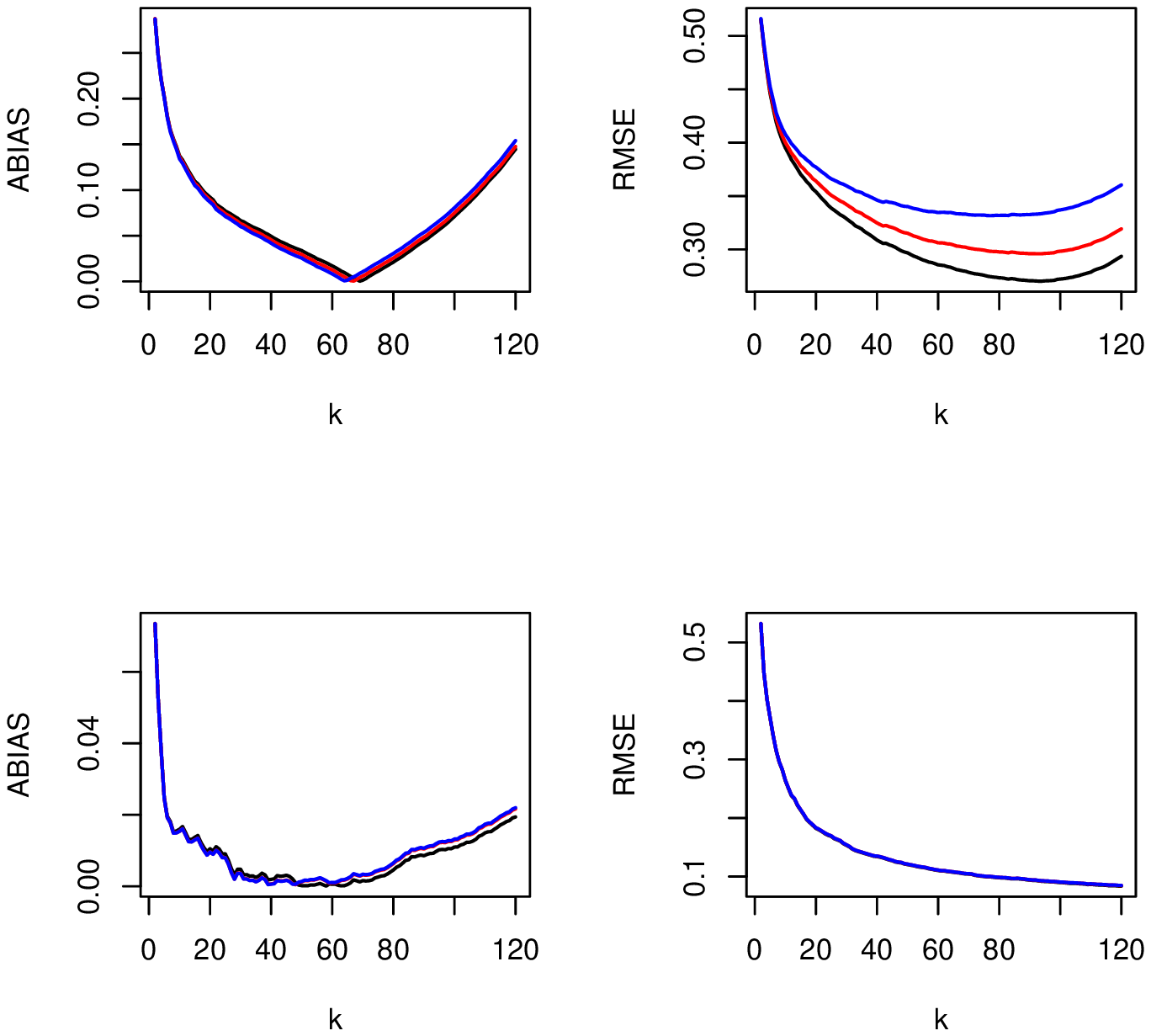}%
\caption{Absolute bias (left two panels) and RMSE (right two panels) of
$\widehat{\gamma}_{1}$ (black) and $\widehat{\gamma}_{1}^{\left(
\mathbf{BMN}\right)  }$ (red) and $\widehat{\gamma}_{1}^{\left(
\mathbf{W}\right)  }$(blue), corresponding to two situations of scenario
$S_{1}:\left(  \gamma_{1}=0.8,p=55\%\right)  $ and $\left(  \gamma
_{1}=0.8,p=90\%\right)  $ based on $1000$ samples of size $300.$}%
\label{s2}%
\end{center}
\end{figure}
%EndExpansion
%

%TCIMACRO{\FRAME{ftbpFU}{4.0439in}{3.915in}{0pt}{\Qcb{Absolute bias (left two
%panels) and RMSE (right two panels) of $\widehat{\gamma}_{1}$ (black) and
%$\widehat{\gamma}_{1}^{\left(  \QTR{bf}{BMN}\right)  }$ (red) and
%$\widehat{\gamma}_{1}^{\left(  \QTR{bf}{W}\right)  }$(blue), corresponding to
%two situations of scenario $S_{2}:\left(  \gamma_{1}=0.6,p=55\%\right)  $ and
%$\left(  \gamma_{1}=0.6,p=90\%\right)  $ based on $1000$ samples of size
%$300.$}}{\Qlb{s3}}{sinario2(g1=0.6).eps}%
%{\special{ language "Scientific Word";  type "GRAPHIC";  display "USEDEF";
%valid_file "F";  width 4.0439in;  height 3.915in;  depth 0pt;
%original-width 5.7648in;  original-height 5.7493in;  cropleft "0";
%filename 'Sinario2(g1=0.6).eps';file-properties "XNPEU";}}}%
%BeginExpansion
\begin{figure}
[ptb]
\begin{center}
\includegraphics[
height=3.915in,
width=4.0439in
]%
{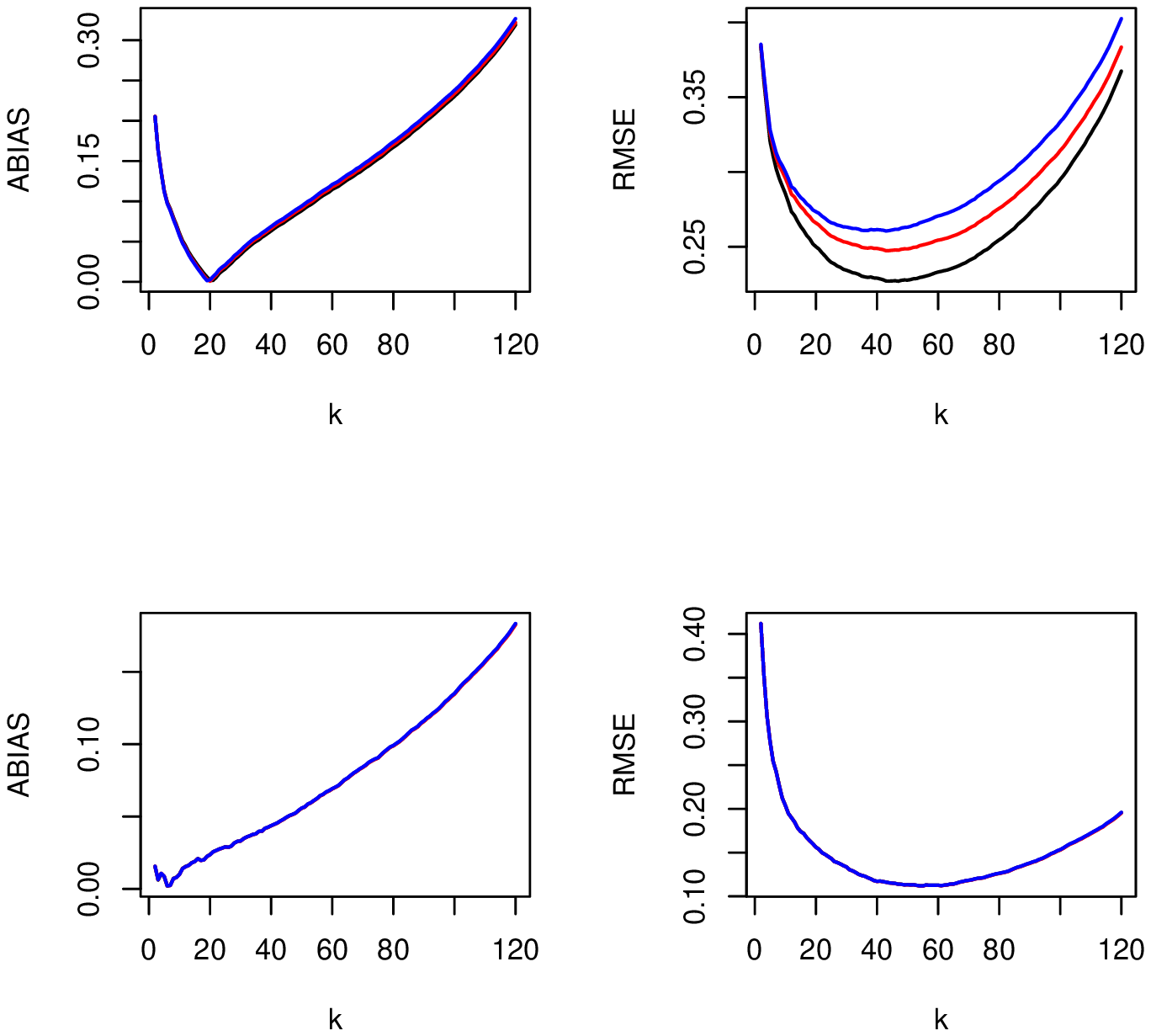}%
\caption{Absolute bias (left two panels) and RMSE (right two panels) of
$\widehat{\gamma}_{1}$ (black) and $\widehat{\gamma}_{1}^{\left(
\mathbf{BMN}\right)  }$ (red) and $\widehat{\gamma}_{1}^{\left(
\mathbf{W}\right)  }$(blue), corresponding to two situations of scenario
$S_{2}:\left(  \gamma_{1}=0.6,p=55\%\right)  $ and $\left(  \gamma
_{1}=0.6,p=90\%\right)  $ based on $1000$ samples of size $300.$}%
\label{s3}%
\end{center}
\end{figure}
%EndExpansion
%

%TCIMACRO{\FRAME{ftbpFU}{4.0439in}{3.915in}{0pt}{\Qcb{Absolute bias (left two
%panels) and RMSE (right two panels) of $\widehat{\gamma}_{1}$ (black) and
%$\widehat{\gamma}_{1}^{\left(  \QTR{bf}{BMN}\right)  }$ (red) and
%$\widehat{\gamma}_{1}^{\left(  \QTR{bf}{W}\right)  }$(blue), corresponding to
%two situations of scenario $S_{2}:\left(  \gamma_{1}=0.8,p=55\%\right)  $ and
%$\left(  \gamma_{1}=0.8,p=90\%\right)  $ based on $1000$ samples of size
%$300.$}}{\Qlb{s4}}{sinario2(g1=0.8).eps}%
%{\special{ language "Scientific Word";  type "GRAPHIC";  display "USEDEF";
%valid_file "F";  width 4.0439in;  height 3.915in;  depth 0pt;
%original-width 6.8044in;  original-height 6.787in;  cropleft "0";
%filename 'sinario2(g1=0.8).eps';file-properties "XNPEU";}}}%
%BeginExpansion
\begin{figure}
[ptb]
\begin{center}
\includegraphics[
height=3.915in,
width=4.0439in
]%
{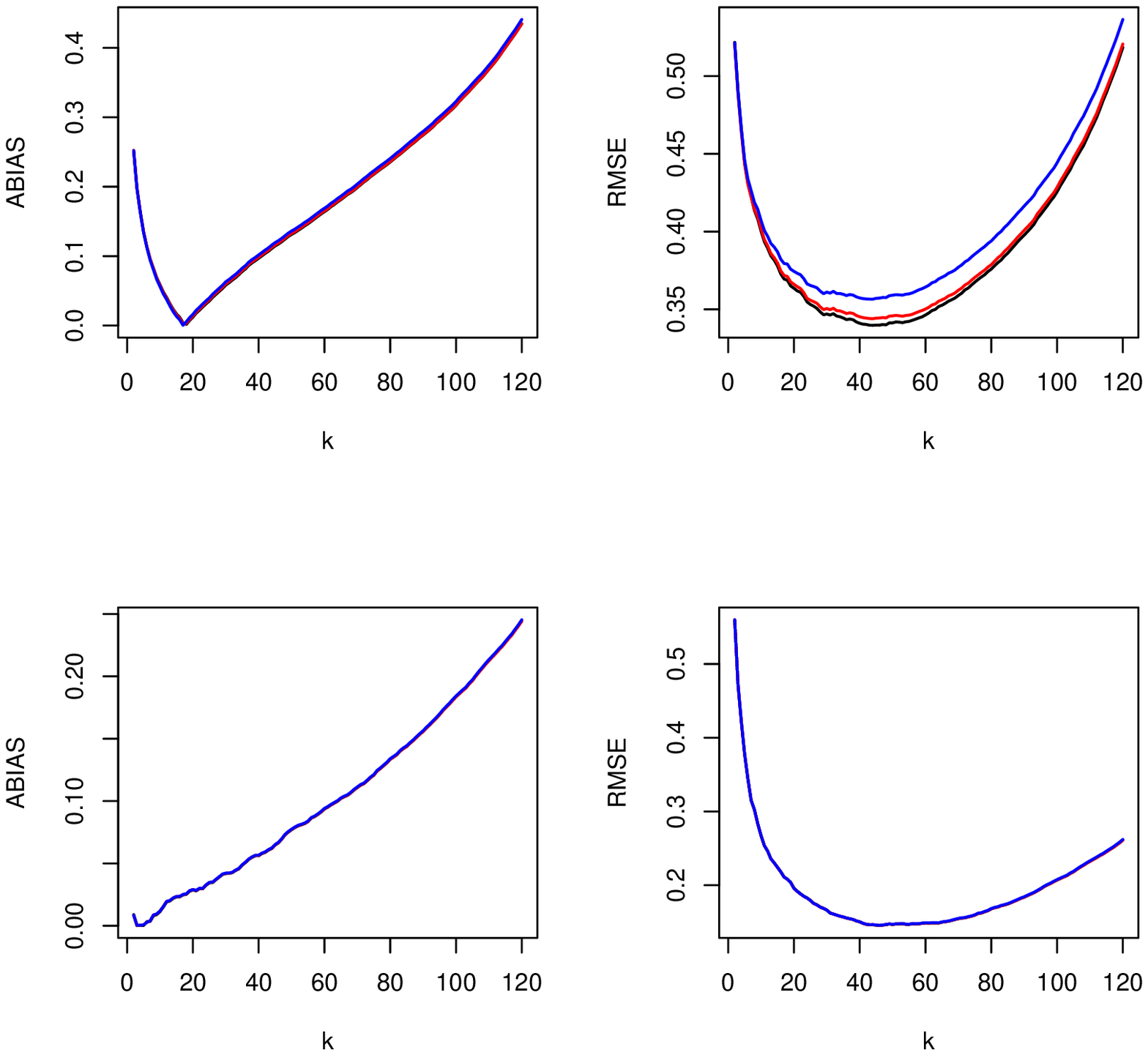}%
\caption{Absolute bias (left two panels) and RMSE (right two panels) of
$\widehat{\gamma}_{1}$ (black) and $\widehat{\gamma}_{1}^{\left(
\mathbf{BMN}\right)  }$ (red) and $\widehat{\gamma}_{1}^{\left(
\mathbf{W}\right)  }$(blue), corresponding to two situations of scenario
$S_{2}:\left(  \gamma_{1}=0.8,p=55\%\right)  $ and $\left(  \gamma
_{1}=0.8,p=90\%\right)  $ based on $1000$ samples of size $300.$}%
\label{s4}%
\end{center}
\end{figure}
%EndExpansion
%

%TCIMACRO{\FRAME{ftbpFU}{4.0568in}{3.915in}{0pt}{\Qcb{Absolute bias (left two
%panels) and RMSE (right two panels) of $\widehat{\gamma}_{1}$ (black) and
%$\widehat{\gamma}_{1}^{\left(  \QTR{bf}{MBN}\right)  }$ (red) and
%$\widehat{\gamma}_{1}^{\left(  \QTR{bf}{W}\right)  }$(blue), corresponding to
%two situations of scenario $S_{3}:\left(  \gamma_{1}=0.6,p=55\%\right)  $ and
%$\left(  \gamma_{1}=0.6,p=90\%\right)  $ based on $1000$ samples of size
%$300.$}}{\Qlb{s5}}{sinario3(g1=0.6).eps}%
%width 4.0568in;  height 3.915in;  depth 0pt;  original-width 5.6239in;
%original-height 5.7519in;  cropleft "0";  croptop "1";  cropright "1.0597";
%cropbottom "0";  filename 'sinario3(g1=0.6).eps';file-properties "XNPEU";}}}%
%BeginExpansion
\begin{figure}
[ptb]
\begin{center}
\includegraphics[
trim=0.000000in 0.000000in -0.335747in 0.000000in,
height=3.915in,
width=4.0568in
]%
{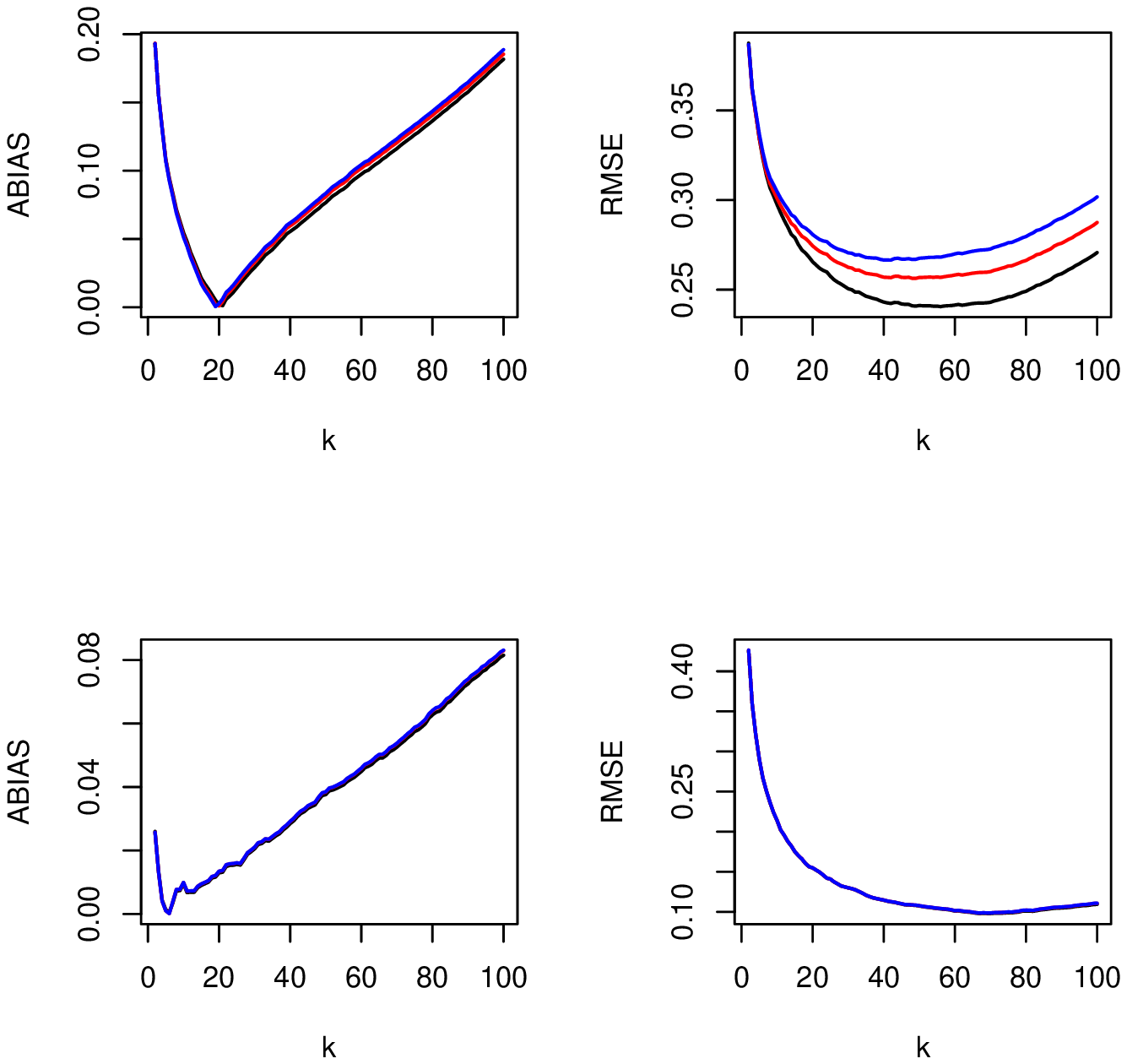}%
\caption{Absolute bias (left two panels) and RMSE (right two panels) of
$\widehat{\gamma}_{1}$ (black) and $\widehat{\gamma}_{1}^{\left(
\mathbf{MBN}\right)  }$ (red) and $\widehat{\gamma}_{1}^{\left(
\mathbf{W}\right)  }$(blue), corresponding to two situations of scenario
$S_{3}:\left(  \gamma_{1}=0.6,p=55\%\right)  $ and $\left(  \gamma
_{1}=0.6,p=90\%\right)  $ based on $1000$ samples of size $300.$}%
\label{s5}%
\end{center}
\end{figure}
%EndExpansion
%

%TCIMACRO{\FRAME{ftbpFU}{4.0439in}{3.915in}{0pt}{\Qcb{Absolute bias (left two
%panels) and RMSE (right two panels) of $\widehat{\gamma}_{1}$ (black) and
%$\widehat{\gamma}_{1}^{\left(  \QTR{bf}{BMN}\right)  }$ (red) and
%$\widehat{\gamma}_{1}^{\left(  \QTR{bf}{W}\right)  }$(blue), corresponding to
%two situations of scenario $S_{3}:\left(  \gamma_{1}=0.8,p=55\%\right)  $ and
%$\left(  \gamma_{1}=0.8,p=90\%\right)  $ based on $1000$ samples of size
%$300.$}}{\Qlb{s6}}{sinario3(g1=0.8).eps}%
%width 4.0439in;  height 3.915in;  depth 0pt;  original-width 5.6516in;
%original-height 5.7648in;  cropleft "0";  croptop "1";  cropright "1.0539";
%cropbottom "0";  filename 'sinario3(g1=0.8).eps';file-properties "XNPEU";}}}%
%BeginExpansion
\begin{figure}
[ptb]
\begin{center}
\includegraphics[
trim=0.000000in 0.000000in -0.304621in 0.000000in,
height=3.915in,
width=4.0439in
]%
{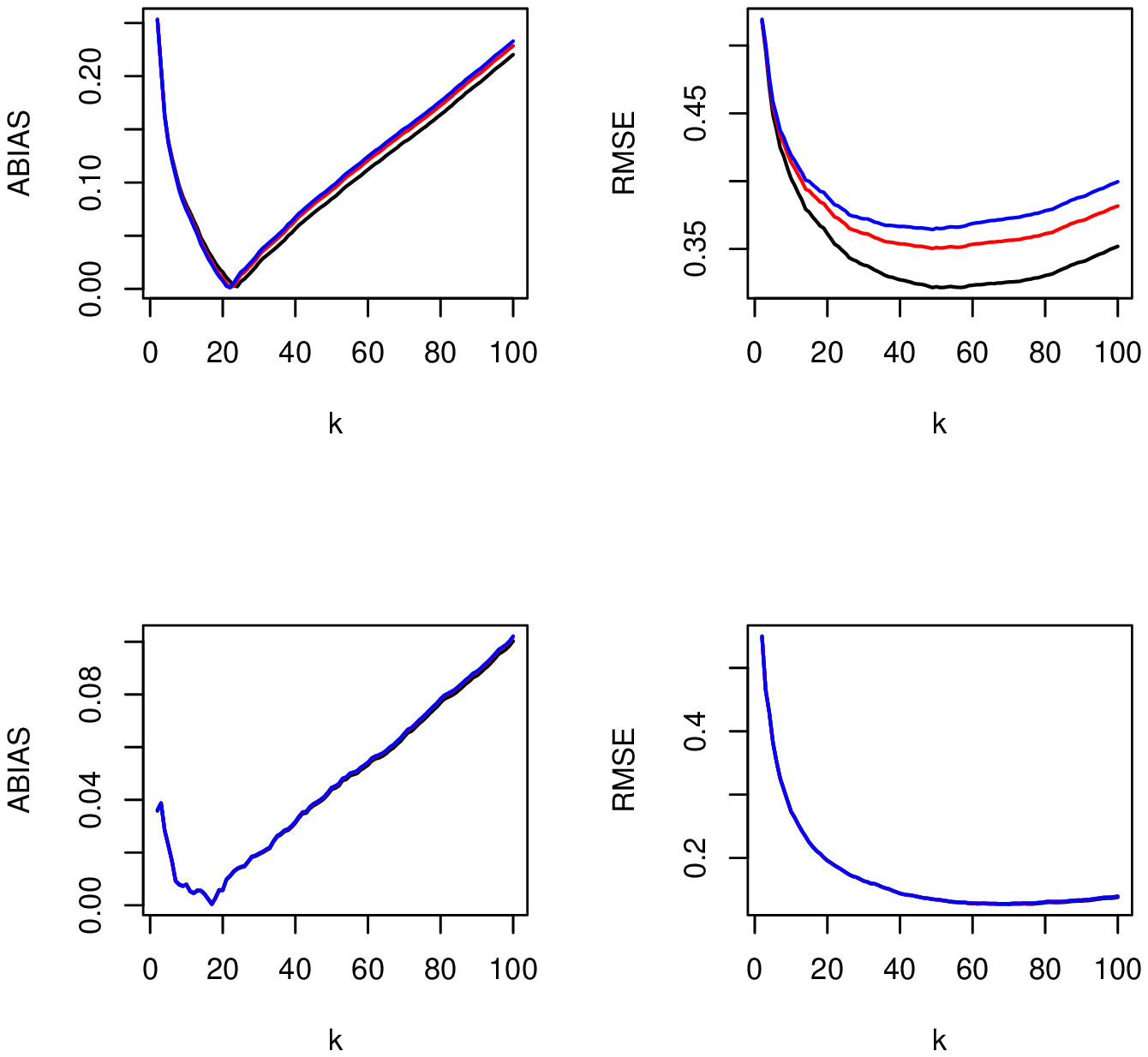}%
\caption{Absolute bias (left two panels) and RMSE (right two panels) of
$\widehat{\gamma}_{1}$ (black) and $\widehat{\gamma}_{1}^{\left(
\mathbf{BMN}\right)  }$ (red) and $\widehat{\gamma}_{1}^{\left(
\mathbf{W}\right)  }$(blue), corresponding to two situations of scenario
$S_{3}:\left(  \gamma_{1}=0.8,p=55\%\right)  $ and $\left(  \gamma
_{1}=0.8,p=90\%\right)  $ based on $1000$ samples of size $300.$}%
\label{s6}%
\end{center}
\end{figure}
%EndExpansion
%

%TCIMACRO{\FRAME{ftbpFU}{4.0439in}{3.915in}{0pt}{\Qcb{Absolute bias (left two
%panels) and RMSE (right two panels) of $\widehat{\gamma}_{1}$ (black) and
%$\widehat{\gamma}_{1}^{\left(  \QTR{bf}{BMN}\right)  }$ (red) and
%$\widehat{\gamma}_{1}^{\left(  \QTR{bf}{W}\right)  }$(blue), corresponding to
%two situations of scenario $S_{4}:\left(  \gamma_{1}=0.6,p=55\%\right)  $ and
%$\left(  \gamma_{1}=0.6,p=90\%\right)  $ based on $1000$ samples of size
%$300.$}}{\Qlb{s7}}{sinario4(g1=0.6).eps}%
%width 4.0439in;  height 3.915in;  depth 0pt;  original-width 5.7648in;
%original-height 5.7493in;  cropleft "0";  croptop "1";  cropright "1.0305";
%cropbottom "0";  filename 'sinario4(g1=0.6).eps';file-properties "XNPEU";}}}%
%BeginExpansion
\begin{figure}
[ptb]
\begin{center}
\includegraphics[
trim=0.000000in 0.000000in -0.175827in 0.000000in,
height=3.915in,
width=4.0439in
]%
{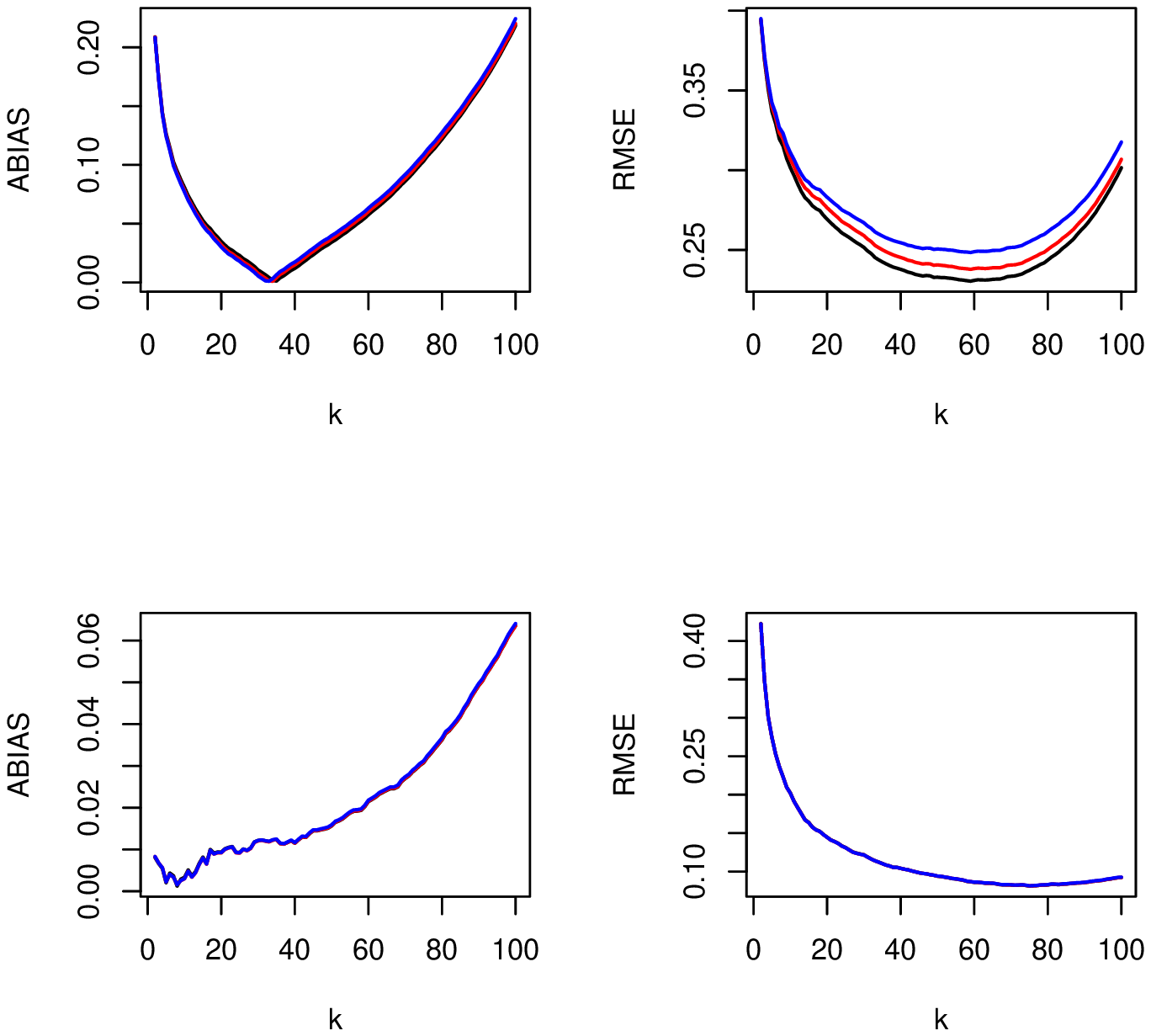}%
\caption{Absolute bias (left two panels) and RMSE (right two panels) of
$\widehat{\gamma}_{1}$ (black) and $\widehat{\gamma}_{1}^{\left(
\mathbf{BMN}\right)  }$ (red) and $\widehat{\gamma}_{1}^{\left(
\mathbf{W}\right)  }$(blue), corresponding to two situations of scenario
$S_{4}:\left(  \gamma_{1}=0.6,p=55\%\right)  $ and $\left(  \gamma
_{1}=0.6,p=90\%\right)  $ based on $1000$ samples of size $300.$}%
\label{s7}%
\end{center}
\end{figure}
%EndExpansion
%

%TCIMACRO{\FRAME{ftbpFU}{4.0439in}{3.915in}{0pt}{\Qcb{Absolute bias (left two
%panels) and RMSE (right two panels) of $\widehat{\gamma}_{1}$ (black) and
%$\widehat{\gamma}_{1}^{\left(  \QTR{bf}{BMN}\right)  }$ (red) and
%$\widehat{\gamma}_{1}^{\left(  \QTR{bf}{W}\right)  }$(blue), corresponding to
%two situations of scenario $S_{4}:\left(  \gamma_{1}=0.8,p=55\%\right)  $ and
%$\left(  \gamma_{1}=0.8,p=90\%\right)  $ based on $1000$ samples of size
%$300.$}}{\Qlb{s8}}{sinario4(g1=0.8).eps}%
%width 4.0439in;  height 3.915in;  depth 0pt;  original-width 5.7648in;
%original-height 5.7493in;  cropleft "0";  croptop "1";  cropright "1.0305";
%cropbottom "0";  filename 'sinario4(g1=0.8).eps';file-properties "XNPEU";}}}%
%BeginExpansion
\begin{figure}
[ptb]
\begin{center}
\includegraphics[
trim=0.000000in 0.000000in -0.175827in 0.000000in,
height=3.915in,
width=4.0439in
]%
{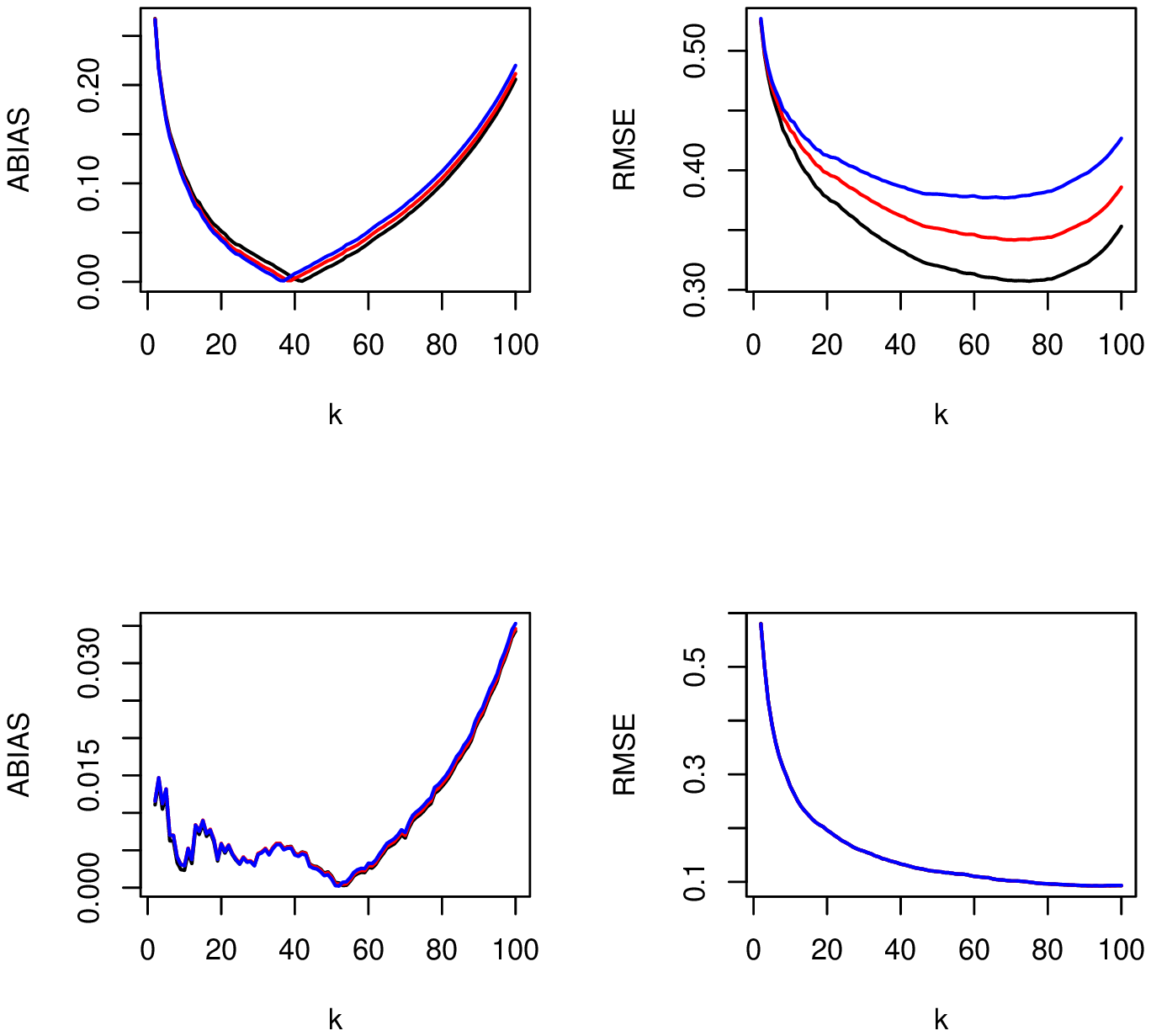}%
\caption{Absolute bias (left two panels) and RMSE (right two panels) of
$\widehat{\gamma}_{1}$ (black) and $\widehat{\gamma}_{1}^{\left(
\mathbf{BMN}\right)  }$ (red) and $\widehat{\gamma}_{1}^{\left(
\mathbf{W}\right)  }$(blue), corresponding to two situations of scenario
$S_{4}:\left(  \gamma_{1}=0.8,p=55\%\right)  $ and $\left(  \gamma
_{1}=0.8,p=90\%\right)  $ based on $1000$ samples of size $300.$}%
\label{s8}%
\end{center}
\end{figure}
%EndExpansion
%

%TCIMACRO{\TeXButton{B}{\begin{table}[h] \centering}}%
%BeginExpansion
\begin{table}[h] \centering
%EndExpansion
$%
\begin{tabular}
[c]{c|cc||cc||cc}
& $k^{\ast}$ & $\widehat{\gamma}_{1}$ & $k^{\ast}$ & $\widehat{\gamma}%
_{1}^{\left(  \mathbf{BMN}\right)  }$ & $k^{\ast}$ & $\widehat{\gamma}%
_{1}^{\left(  \mathbf{W}\right)  }$\\\hline
$S1$ & $44$ & $0.600$ & $41$ & $0.599$ & $40$ & $0.600$\\\hline
$S2$ & $18$ & $0.601$ & $17$ & $0.600$ & $16$ & $0.597$\\\hline
$S3$ & $21$ & $0.601$ & $20$ & $0.601$ & $19$ & $0.599$\\\hline
$S4$ & $30$ & $0.603$ & $27$ & $0.600$ & $25$ & $0.598$\\\hline
\end{tabular}
$%
\caption{Optimal sample fraction $k^{\ast}$ and the estimated value of each estimator of the tail index  $\gamma_{1}=0.6$  based on 1000 samples for the four scenarios with $p=0.55$ .
}\label{tab1}%
%TCIMACRO{\TeXButton{E}{\end{table}}}%
%BeginExpansion
\end{table}%
%EndExpansion
%

%TCIMACRO{\TeXButton{B}{\begin{table}[h] \centering}}%
%BeginExpansion
\begin{table}[h] \centering
%EndExpansion
$%
\begin{tabular}
[c]{c|cc||cc||cc}
& $k^{\ast}$ & $\widehat{\gamma}_{1}$ & $k^{\ast}$ & $\widehat{\gamma}%
_{1}^{\left(  \mathbf{BMN}\right)  }$ & $k^{\ast}$ & $\widehat{\gamma}%
_{1}^{\left(  \mathbf{W}\right)  }$\\\hline
$S1$ & $82$ & $0.610$ & $82$ & $0.611$ & $82$ & $0.611$\\\hline
$S2$ & $37$ & $0.640$ & $37$ & $0.640$ & $37$ & $0.640$\\\hline
$S3$ & $46$ & $0.633$ & $37$ & $0.625$ & $37$ & $0.625$\\\hline
$S4$ & $52$ & $0.610$ & $52$ & $0.610$ & $52$ & $0.610$\\\hline
\end{tabular}
$%
\caption{Optimal sample fraction $k^{\ast}$ and the estimated value of each estimator of the tail index  $\gamma_{1}=0.6$  based on 1000 samples for  the four scenarios with $p=0.9$.
}\label{tab2}%
%TCIMACRO{\TeXButton{E}{\end{table}}}%
%BeginExpansion
\end{table}%
%EndExpansion
%

%TCIMACRO{\TeXButton{B}{\begin{table}[h] \centering}}%
%BeginExpansion
\begin{table}[h] \centering
%EndExpansion
$%
\begin{tabular}
[c]{c|cc||cc||cc}
& $k^{\ast}$ & $\widehat{\gamma}_{1}$ & $k^{\ast}$ & $\widehat{\gamma}%
_{1}^{\left(  \mathbf{BMN}\right)  }$ & $k^{\ast}$ & $\widehat{\gamma}%
_{1}^{\left(  \mathbf{W}\right)  }$\\\hline
$S1$ & $59$ & $0.799$ & $57$ & $0.800$ & $54$ & $0.799$\\\hline
$S2$ & $21$ & $0.803$ & $21$ & $0.803$ & $20$ & $0.799$\\\hline
$S3$ & $24$ & $0.802$ & $22$ & $0.798$ & $22$ & $0.801$\\\hline
$S4$ & $51$ & $0.799$ & $52$ & $0.800$ & $50$ & $0.801$\\\hline
\end{tabular}
$%
\caption{Optimal sample fraction $k^{\ast}$ and the estimated value of each estimator of the tail index  $\gamma_{1}=0.8$  based on 1000 samples for the four scenarios with $p=0.55$ .
}\label{tab3}%
%TCIMACRO{\TeXButton{E}{\end{table}}}%
%BeginExpansion
\end{table}%
%EndExpansion
%

%TCIMACRO{\TeXButton{B}{\begin{table}[h] \centering}}%
%BeginExpansion
\begin{table}[h] \centering
%EndExpansion
$%
\begin{tabular}
[c]{c|cc||cc||cc}
& $k^{\ast}$ & $\widehat{\gamma}_{1}$ & $k^{\ast}$ & $\widehat{\gamma}%
_{1}^{\left(  \mathbf{BMN}\right)  }$ & $k^{\ast}$ & $\widehat{\gamma}%
_{1}^{\left(  \mathbf{W}\right)  }$\\\hline
$S1$ & $90$ & $0.804$ & $90$ & $0.806$ & $90$ & $0.807$\\\hline
$S2$ & $34$ & $0.845$ & $34$ & $0.846$ & $34$ & $0.846$\\\hline
$S3$ & $40$ & $0.831$ & $40$ & $0.831$ & $40$ & $0.831$\\\hline
$S4$ & $71$ & $0.814$ & $71$ & $0.814$ & $71$ & $0.815$\\\hline
\end{tabular}
$%
\caption{Optimal sample fraction $k^{\ast}$ and the estimated value of each estimator of the tail index   $\gamma_{1}=0.8$  based on 1000 samples for  the four scenarios with $p=0.9$.
}\label{tab4}%
%TCIMACRO{\TeXButton{E}{\end{table}}}%
%BeginExpansion
\end{table}%
%EndExpansion

\end{center}

\section{\textbf{Real data example\label{sec4}}}

\noindent In this section, we give an application to the AIDS data set,
available in the "DTDA" R package, used before by \cite{Lagakos88}. The data
present the infection and induction times for $n=258$ adults who were infected
with HIV virus and developed AIDS by June $30,1986.$ The time in years,
measured from April $1,1978,$ when adults were infected by the virus from a
contaminated blood transfusion and the waiting time to development of AIDS
measured from the date of infection. We interest here to the estimation of the
end-time of the induction of the AIDS virus which corresponds to the high
quantile in basis the given observations. The variable of interest here is the
time of induction $\mathbf{T}$ of the disease duration which elapses between
the date of infection $M$ and the date $M+T$ of the declaration of the
disease. The sample $(T_{1},M_{1}),...,(T_{n},M_{n})$ are taken between two
fixed dates: $"0"$ and $"8"$, i.e. between April $1,1978,$ and June $30,1986.$
The initial date $"0"$ denotes an infection occurring in the three months:
from April $1,1978$, to June $30,1978.$ Let us assume that $M$ and $T$ are the
observed rv's, corresponding to the underlying rv's $\mathbf{M}$ and
$\mathbf{T},$ given by the truncation scheme $0\leq M+T\leq8,$ which in turn
may be rewritten into%
\begin{equation}
0\leq M\leq S, \label{MT}%
\end{equation}
where $S:=8-T.$ To work within the framework of the present paper, let us make
the following transformations:%
\begin{equation}
X:=\frac{1}{S+\epsilon}\text{ and }Y:=\frac{1}{M+\epsilon}, \label{tr}%
\end{equation}
where $\epsilon=0.05$ so that the two denominators be non-null. Thus, in view
of $\left(  \ref{MT}\right)  $, we have $X\leq Y,$\ which means that $X$ is
randomly right-truncated by $Y.$ Thereby, for the given sample $(T_{1}%
,M_{1}),...,(T_{n},M_{n})$, from $\left(  T,M\right)  ,$ the previous
transformations produce us a new ones $(X_{1},Y_{1}),...,(X_{n},Y_{n})$ from
$\left(  X,Y\right)  .$\medskip\

\noindent Let us now denote by $\mathbf{F}$ and $\mathbf{G}$ the df's of the
underling rv's $\mathbf{X}$ and $\mathbf{Y}$ corresponding to the truncated
rv's $X$ and $Y,$ respectively. By using parametric likelihood methods,
\cite{L86} fit both df's of $\mathbf{M}$ and $\mathbf{S}$ by the two-parameter
Weibull model, this implies that the df's of $\mathbf{F}$ and $\mathbf{G}$ by
may be fitted by two-parameter Fr\'{e}chet model, namely $\mathbf{H}_{\left(
a.r\right)  }\left(  x\right)  =\exp\left(  -a^{r}x^{-r}\right)  ,$ $x>0,$
$a>0,$ $r>0,$ hence both $\mathbf{F}$ and $\mathbf{G}$ are heavy-tailed. The
estimated parameters corresponding to the fitting of df $\mathbf{G}$ are
$a_{0}=0.004$ and $r_{0}=2.1,$ see also \cite{Lagakos88} page 520. Thus on may
consider that df $\mathbf{G}$ is known and equals $\mathbf{G}_{\theta
}=\mathbf{H}_{\left(  a_{0},r_{0}\right)  },$ where $\theta=\left(
a_{0},r_{0}\right)  .$ By using the Thomas and Reiss algorithm, given above,
we compute the optimal sample fraction $k^{\ast}$ corresponds to the tail
index estimator $\widehat{\gamma}_{1}$ of df $\mathbf{F}$ is $\gamma_{1}.$ We
find%
\begin{equation}
k^{\ast}=19,\text{ }X_{n-k:n}=0.356\text{ and }\widehat{\gamma}_{1}=0.917.
\label{v}%
\end{equation}
The well-known Weissman estimator \citep[][]{W78} of the high quantile,
$q_{v}:=\mathbf{F}^{-1}\left(  1-v_{n}\right)  ,$ corresponding to the
underling df $\mathbf{F}$ is given by%
\[
\widehat{q}_{v}:=X_{n-k:n}\left(  \frac{v}{\overline{\mathbf{F}}_{n}\left(
X_{n-k:n}\right)  }\right)  ^{-\widehat{\gamma}_{1}},
\]
where $v=1/\left(  2n\right)  ,$ and $\mathbf{F}_{n}$ is the semiparametric
estimator of df $\mathbf{F}$ of $\mathbf{X}$ given in $\left(  \ref{SCMLE}%
\right)  .$ From the values $\left(  \ref{v}\right)  ,$ we get $\widehat
{q}_{v}=0.061.$ Let us now compute the high quantile of $\mathbf{T}$ based on
the original data, $T_{1},...,T_{n}.$ Recall that $\mathbf{P}\left(
\mathbf{X}\geq q_{v}\right)  =v$ and $\mathbf{X=1/}\left(  8-\mathbf{T}%
+\epsilon\right)  ,$ this implies that $\mathbf{P}\left(  \mathbf{T}%
\geq1/q_{v}-8+\epsilon\right)  =v,$ this means that $1/q_{v}-8+\epsilon$ is
the high quantile of $\mathbf{T},$ which corresponds to the end-time $t_{end}$
that we want to estimate. Thereby $\widehat{t}_{end}=1/\widehat{q}%
_{v}-8+10^{-2}=1/0.061-8+10^{-2}=8.40,$ the value the end time of induction of
AIDS is: $8$ years, $4$ months and $24$ days.

\section{\textbf{Proof of Theorems\label{sec5}}}

\subsection{Proof of Theorem \ref{Theorem1}.}

Let us first notice that the semiparametric estimator of df $\mathbf{F}$ given
in $\left(  \ref{CLME}\right)  $ may be rewritten into%
\begin{equation}
\mathbf{F}_{n}\left(  x;\widehat{\theta}_{n}\right)  =P_{n}\left(
\widehat{\theta}\right)  \int_{0}^{x}\frac{dF_{n}^{\ast}\left(  w\right)
}{\overline{\mathbf{G}}_{\widehat{\theta}}\left(  w\right)  }, \label{Fn-rep}%
\end{equation}
and $1/P_{n}\left(  \widehat{\theta}\right)  =\int_{0}^{\infty}dF_{n}^{\ast
}\left(  w\right)  /\overline{\mathbf{G}}_{\widehat{\theta}}\left(  w\right)
,$ where $F_{n}^{\ast}\left(  w\right)  :=n^{-1}\sum_{i=1}^{n}\mathbf{1}%
\left(  X_{i}\leq w\right)  $ denotes the usual empirical df pertaining to the
observed sample $X_{1},...,X_{n}.$ It is worth mentioning that by using the
strong law of large numbers $P_{n}\left(  \widehat{\theta}\right)  \rightarrow
P\left(  \theta\right)  $ (almost surely) as $n\rightarrow\infty,$ where
$P\left(  \theta\right)  =1/\int_{0}^{\infty}dF^{\ast}\left(  w\right)
/\overline{\mathbf{G}}_{\theta}\left(  w\right)  $ (see e.g. Lemma 3.2 in
\cite{Wang89}). On the other hand by using the first equation in $\left(
\ref{F-G}\right)  ,$ we deduce that $p=1/\int_{0}^{\infty}dF^{\ast}\left(
w\right)  /\overline{\mathbf{G}}\left(  w\right)  ,$ it follows that $p\equiv
P\left(  \theta\right)  $ because we already assumed that $\mathbf{G\equiv
G}_{\theta}.$ Next we use the distribution tail
\begin{equation}
\overline{\mathbf{F}}\left(  x;\theta\right)  =P\left(  \theta\right)
\int_{x}^{\infty}\frac{dF^{\ast}\left(  w\right)  }{\overline{\mathbf{G}%
}_{\theta}\left(  w\right)  }, \label{F-rep}%
\end{equation}
and its empirical counterpart%
\[
\overline{\mathbf{F}}_{n}\left(  x;\widehat{\theta}\right)  =P_{n}\left(
\widehat{\theta}\right)  \int_{x}^{\infty}\frac{dF_{n}^{\ast}\left(  w\right)
}{\overline{\mathbf{G}}_{\widehat{\theta}}\left(  w\right)  }.
\]
To begin let us decompose $k^{-1/2}\mathbf{D}_{n}\left(  x;\widehat{\theta
}\right)  $ for $x>1$ into the sum of%
\[
\mathbf{M}_{n1}\left(  x\right)  :=x^{-1/\gamma_{1}}\frac{\overline
{\mathbf{F}}_{n}\left(  xX_{n-k:n};\widehat{\theta}\right)  -\overline
{\mathbf{F}}_{n}\left(  xX_{n-k:n};\theta\right)  }{\overline{\mathbf{F}%
}\left(  xX_{n-k:n};\theta\right)  },
\]%
\[
\mathbf{M}_{n2}\left(  x\right)  :=x^{-1/\gamma_{1}}\frac{\overline
{\mathbf{F}}_{n}\left(  xX_{n-k:n};\theta\right)  -\overline{\mathbf{F}%
}\left(  xX_{n-k:n};\theta\right)  }{\overline{\mathbf{F}}\left(
xX_{n-k:n};\theta\right)  },
\]%
\[
\mathbf{M}_{n3}\left(  x\right)  :=-\frac{\overline{\mathbf{F}}\left(
xX_{n-k:n};\theta\right)  }{\overline{\mathbf{F}}_{n}\left(  X_{n-k:n}%
;\theta\right)  }\frac{\overline{\mathbf{F}}_{n}\left(  X_{n-k:n}%
;\theta\right)  -\overline{\mathbf{F}}\left(  X_{n-k:n};\theta\right)
}{\overline{\mathbf{F}}\left(  X_{n-k:n};\theta\right)  },
\]%
\[
\mathbf{M}_{n4}\left(  x\right)  :=\left(  \frac{\overline{\mathbf{F}}\left(
xX_{n-k:n};\theta\right)  }{\overline{\mathbf{F}}_{n}\left(  X_{n-k:n}%
;\theta\right)  }-x^{-1/\gamma_{1}}\right)  \frac{\overline{\mathbf{F}}%
_{n}\left(  xX_{n-k:n};\theta\right)  -\overline{\mathbf{F}}\left(
xX_{n-k:n};\theta\right)  }{\overline{\mathbf{F}}\left(  xX_{n-k:n}%
;\theta\right)  }%
\]
and%
\[
\mathbf{M}_{n5}\left(  x\right)  :=\frac{\overline{\mathbf{F}}\left(
xX_{n-k:n};\theta\right)  }{\overline{\mathbf{F}}\left(  X_{n-k:n}%
;\theta\right)  }-x^{-1/\gamma_{1}}.
\]
Our goal is to provide a weighted weak approximation to the tail empirical
process $\mathbf{D}_{n}\left(  x;\widehat{\theta};\gamma_{1}\right)  .$ To
begin, let $\xi_{i}:=\overline{F}^{\ast}\left(  X_{i}\right)  ,$ $i=1,...,n$
be a sequence of independent and identically rv's. Recall that both df's
$\mathbf{F}$ and $\mathbf{G}_{\theta}$ are assumed to be continuous, this
implies that $F^{\ast}$ is continuous as well, therefore $\mathbf{P}\left(
\xi_{i}\leq u\right)  =u,$ this means that $\left(  \xi_{i}\right)  _{i=1,n}$
are uniformly distributed on $\left(  0,1\right)  .$ Let us now define the
corresponding uniform tail empirical process
\begin{equation}
\alpha_{n}\left(  s\right)  :=\sqrt{k}\left(  \mathbf{U}_{n}\left(  s\right)
-s\right)  ,\text{ for }0\leq s\leq1, \label{alfan}%
\end{equation}
where
\begin{equation}
\mathbf{U}_{n}\left(  s\right)  :=k^{-1}\sum_{i=1}^{n}\mathbf{1}\left(
\xi_{i}<ks/n\right)  \label{tail}%
\end{equation}
denotes the tail empirical df pertaining to the sample $\left(  \xi
_{i}\right)  _{i=1,n}.$ In view of Proposition 3.1 of \cite{Einmahl2006},
there exists a Wiener process $W$ such that for every $0\leq\epsilon<1/2,$
\begin{equation}
\sup_{0\leq s<1}s^{-\epsilon}\left\vert \alpha_{n}\left(  s\right)  -W\left(
s\right)  \right\vert \overset{\mathbf{P}}{\rightarrow}0,\text{ as
}n\rightarrow\infty. \label{weak-alfa}%
\end{equation}
Let us fix a sufficiently small $0<\epsilon<1/2.$ We will successively show
that, under the first-order of regular variation conditions $\left(
\ref{rv1}\right)  $, uniformly on $x\geq1,$ for all large $n:$%
\begin{equation}
\sqrt{k}\mathbf{M}_{n2}\left(  x\right)  =\frac{\gamma}{\gamma_{1}}%
x^{1/\gamma_{2}}W\left(  t^{-1/\gamma}\right)  +\frac{\gamma}{\gamma_{1}}%
\int_{x^{1/\gamma_{2}}}^{\infty}W\left(  t^{-\gamma_{2}/\gamma}\right)
dt+o_{\mathbf{P}}\left(  x^{\frac{1}{2}\left(  \frac{1}{\gamma_{2}}-\frac
{1}{\gamma_{1}}\right)  +\epsilon}\right)  , \label{approx-M2}%
\end{equation}
and%
\begin{equation}
\sqrt{k}\mathbf{M}_{n3}\left(  x\right)  =-x^{-1/\gamma_{1}}\left(
\frac{\gamma}{\gamma_{1}}W\left(  1\right)  +\frac{\gamma}{\gamma_{1}}\int
_{1}^{\infty}W\left(  t^{-\gamma_{2}/\gamma}\right)  dt\right)  +o_{\mathbf{P}%
}\left(  x^{-1/\gamma_{1}+\epsilon}\right)  . \label{approx-M3}%
\end{equation}
while%
\begin{equation}
\sqrt{k}\mathbf{M}_{n1}\left(  x\right)  =o_{\mathbf{P}}\left(  x^{-1/\gamma
_{1}+\epsilon}\right)  ,\text{ }\sqrt{k}\mathbf{M}_{n4}\left(  x\right)
=o_{\mathbf{P}}\left(  x^{\frac{1}{2}\left(  \frac{1}{\gamma_{2}}-\frac
{1}{\gamma_{1}}\right)  +\epsilon}\right)  ,\text{ } \label{res-M14}%
\end{equation}
and%
\begin{equation}
\sqrt{k}\mathbf{M}_{n5}\left(  x\right)  =x^{-1/\gamma_{1}}\dfrac{x^{\rho
_{1}/\gamma_{1}}-1}{\rho_{1}\gamma_{1}}\sqrt{k}\mathbf{A}\left(  a_{k}\right)
+o_{\mathbf{P}}\left(  x^{-1/\gamma_{1}}\right)  . \label{res-M5}%
\end{equation}
Throughout the proof, without loss of generality, we assume that
$a\epsilon\equiv\epsilon,$ for any constant $a>0.$ We point out that all the
rest terms of the previous approximations are negligible in probability,
uniformly on $x>1.$ Let us begin by the term\textbf{\ }$\mathbf{M}_{n1}\left(
x\right)  $ which may be made into%
\begin{align*}
&  \frac{x^{-1/\gamma_{1}}}{\overline{\mathbf{F}}\left(  xX_{n-k:n}%
;\theta\right)  }P_{n}\left(  \widehat{\theta}\right)  \left(  \int
_{x}^{\infty}\frac{dF_{n}^{\ast}\left(  X_{n-k:n}w\right)  }{\overline
{\mathbf{G}}_{\widehat{\theta}}\left(  X_{n-k:n}w\right)  }-\int_{x}^{\infty
}\frac{dF_{n}^{\ast}\left(  X_{n-k:n}w\right)  }{\overline{\mathbf{G}}%
_{\theta}\left(  X_{n-k:n}w\right)  }\right) \\
&  =\frac{x^{-1/\gamma_{1}}}{\overline{\mathbf{F}}\left(  xX_{n-k:n}%
;\theta\right)  }P_{n}\left(  \widehat{\theta}\right)  \int_{x}^{\infty
}\left(  \frac{1}{\overline{\mathbf{G}}_{\widehat{\theta}}\left(
X_{n-k:n}w\right)  }-\frac{1}{\overline{\mathbf{G}}_{\theta}\left(
X_{n-k:n}w\right)  }\right)  dF_{n}^{\ast}\left(  X_{n-k:n}w\right)  .
\end{align*}
By applying the mean value theorem (for several variables) to function
$\theta\rightarrow1/\overline{\mathbf{G}}_{\theta}\left(  \cdot\right)  ,$
yields%
\[
\frac{1}{\overline{\mathbf{G}}_{\widehat{\theta}}\left(  z\right)  }-\frac
{1}{\overline{\mathbf{G}}_{\theta}\left(  z\right)  }=\sum_{i=1}^{d}\left(
\widehat{\theta}_{i}-\theta_{i}\right)  \frac{\overline{\mathbf{G}%
}_{\widetilde{\theta}}^{\left(  i\right)  }\left(  z\right)  }{\overline
{\mathbf{G}}_{\widetilde{\theta}}^{2}\left(  z\right)  },\text{ for any }z>1,
\]
where $\widetilde{\theta}$ is such that $\widetilde{\theta}_{i}$ is between
$\theta_{i}$ and $\widehat{\theta}_{i},$ for $i=1,...,d,$ therefore%
\[
\mathbf{M}_{n1}\left(  x\right)  =\frac{x^{-1/\gamma_{1}}}{\overline
{\mathbf{F}}\left(  xX_{n-k:n};\theta\right)  }P_{n}\left(  \widehat{\theta
}\right)  \sum_{i=1}^{d}\left(  \widehat{\theta}_{i}-\theta_{i}\right)
\int_{x}^{\infty}\frac{\overline{\mathbf{G}}_{\widetilde{\theta}}^{\left(
i\right)  }\left(  X_{n-k:n}w\right)  }{\overline{\mathbf{G}}_{\widetilde
{\theta}}^{2}\left(  X_{n-k:n}w\right)  }dF_{n}^{\ast}\left(  X_{n-k:n}%
w\right)  .
\]
Recall that by assumptions $\left(  \ref{rv1}\right)  $ and $\left[
A2\right]  $ both $\overline{\mathbf{G}}_{\theta}$ and $\overline{\mathbf{G}%
}_{\theta}^{\left(  i\right)  }$ are regularly varying with the same index
$\left(  -1/\gamma_{2}\right)  $ and on the other hand, $X_{n-k:n}%
\overset{\mathbf{P}}{\rightarrow}\infty$ and $w>1,$ then $X_{n-k:n}%
w\overset{\mathbf{P}}{\rightarrow}\infty.$ Then by applying Pooter's
inequalities $\left(  \ref{pooter}\right)  ,$ yields%
\[
\frac{\overline{\mathbf{G}}_{\widetilde{\theta}}\left(  X_{n-k:n}w\right)
}{\overline{\mathbf{G}}_{\widetilde{\theta}}\left(  X_{n-k:n}\right)
}=\left(  1+o_{\mathbf{P}}\left(  1\right)  \right)  w^{-1/\gamma_{2}%
+\epsilon}=\frac{\overline{\mathbf{G}}_{\widetilde{\theta}}^{\left(  i\right)
}\left(  X_{n-k:n}w\right)  }{\overline{\mathbf{G}}_{\widetilde{\theta}%
}^{\left(  i\right)  }\left(  X_{n-k:n}\right)  },
\]
it follows that%
\begin{align*}
\mathbf{M}_{n1}\left(  x\right)   &  =\left(  1+o_{\mathbf{P}}\left(
1\right)  \right)  P_{n}\left(  \widehat{\theta}\right)  \frac{x^{-1/\gamma
_{1}}}{\overline{\mathbf{G}}_{\widetilde{\theta}}\left(  X_{n-k:n}\right)
\overline{\mathbf{F}}\left(  xX_{n-k:n};\theta\right)  }\\
&  \times\sum_{i=1}^{d}\frac{\overline{\mathbf{G}}_{\widetilde{\theta}%
}^{\left(  i\right)  }\left(  X_{n-k:n}\right)  }{\overline{\mathbf{G}%
}_{\widetilde{\theta}}\left(  X_{n-k:n}\right)  }\left\vert \widehat{\theta
}_{i}-\theta_{i}\right\vert \int_{x}^{\infty}w^{1/\gamma_{2}-\epsilon}%
dF_{n}^{\ast}\left(  X_{n-k:n}w\right)  .
\end{align*}
For some regularity assumptions, \cite{Anders70} stated that $\sqrt{n}\left(
\widehat{\theta}-\theta\right)  $ is asymptotically a centred multivariate
normal rv, which implies that $\widehat{\theta}_{i}-\theta_{i}=O_{\mathbf{P}%
}\left(  n^{-1/2}\right)  $ thus $\widehat{\theta}\overset{\mathbf{P}%
}{\rightarrow}\theta.$ \ On the other hand, by the law of large numbers
$P_{n}\left(  \theta\right)  \overset{\mathbf{P}}{\rightarrow}P\left(
\theta\right)  $ as $n\rightarrow\infty,$ then we may readily show that
$P_{n}\left(  \widehat{\theta}\right)  \overset{\mathbf{P}}{\rightarrow
}P\left(  \theta\right)  $ as $n\rightarrow\infty$ as well. Note since
$\widehat{\theta}$ is consistent estimator for $\theta$ then $\widetilde
{\theta}$ it is, then by using the fact that $X_{n-k:n}\overset{\mathbf{P}%
}{\rightarrow}\infty,$ and the two assumptions $\left[  A1\right]  $ and
$\left[  A3\right]  $ together, we show readily that%
\[
\left(  X_{n-k:n}\right)  ^{-\epsilon}\frac{\overline{\mathbf{G}}%
_{\widetilde{\theta}}^{\left(  i\right)  }\left(  X_{n-k:n}\right)
}{\overline{\mathbf{G}}_{\widetilde{\theta}}\left(  X_{n-k:n}\right)
}\overset{\mathbf{P}}{\rightarrow}0,\text{ as }n\rightarrow\infty,
\]
and $\overline{\mathbf{G}}_{\theta}\left(  X_{n-k:n}\right)  /\overline
{\mathbf{G}}_{\widetilde{\theta}}\left(  X_{n-k:n}\right)  \overset
{\mathbf{P}}{\rightarrow}1.$ In view of Lemma A1 in \cite{BchMN-16a}, we infer
that $X_{n-k:n}=\left(  1+o_{\mathbf{P}}\left(  1\right)  \right)  \left(
k/n\right)  ^{-\gamma},$ thus
\[
\mathbf{M}_{n1}\left(  x\right)  =\left(  k/n\right)  ^{-\epsilon\gamma
}o_{\mathbf{P}}\left(  n^{-1/2}\right)  \widetilde{\mathbf{M}}_{n1}\left(
x\right)  ,
\]
where%
\[
\widetilde{\mathbf{M}}_{n1}\left(  x\right)  :=\frac{x^{-1/\gamma_{1}}P\left(
\theta\right)  }{\overline{\mathbf{G}}_{\theta}\left(  X_{n-k:n}\right)
\overline{\mathbf{F}}\left(  xX_{n-k:n};\theta\right)  }\int_{x}^{\infty
}w^{1/\gamma_{2}-\epsilon}dF_{n}^{\ast}\left(  X_{n-k:n}w\right)  .
\]
Making use of representation $\left(  \ref{F-rep}\right)  ,$ we write
\begin{align}
\widetilde{\mathbf{M}}_{n1}\left(  x\right)   &  =x^{-1/\gamma_{1}}\left(
\int_{x}^{\infty}\frac{\overline{\mathbf{G}}_{\theta}\left(  X_{n-k:n}\right)
}{\overline{\mathbf{G}}_{\theta}\left(  X_{n-k:n}w\right)  }d\frac{F^{\ast
}\left(  X_{n-k:n}w\right)  }{\overline{F}^{\ast}\left(  X_{n-k:n}\right)
}\right)  ^{-1}\label{M1-tild}\\
&  \times\left(  \int_{x}^{\infty}w^{1/\gamma_{2}-\epsilon}d\frac{F_{n}^{\ast
}\left(  X_{n-k:n}w\right)  }{\overline{F}^{\ast}\left(  X_{n-k:n}\right)
}\right)  .\nonumber
\end{align}
Once again by using the routine manipulations of Potter's inequalities, we
show that the first quantity between two brackets is%
\[
\left(  1+o_{\mathbf{P}}\left(  1\right)  \right)  \int_{x}^{\infty
}w^{1/\gamma_{2}+\epsilon/2}d\frac{F^{\ast}\left(  X_{n-k:n}w\right)
}{\overline{F}^{\ast}\left(  X_{n-k:n}w\right)  }.
\]
By using an integration by parts to the previous integral yields%
\[
w^{1/\gamma_{2}+\epsilon/2}\frac{\overline{F}^{\ast}\left(  X_{n-k:n}x\right)
}{\overline{F}^{\ast}\left(  X_{n-k:n}x\right)  }+\left(  1/\gamma
_{2}+\epsilon/2\right)  \int_{x}^{\infty}w^{1/\gamma_{2}+\epsilon/2-1}%
\frac{\overline{F}^{\ast}\left(  X_{n-k:n}x\right)  }{\overline{F}^{\ast
}\left(  X_{n-k:n}x\right)  }dw.
\]
Recall that from$\ \left(  \ref{F-G-stars}\right)  $ we have $\overline
{F}^{\ast}\in\mathcal{RV}_{\left(  -1/\gamma\right)  },$ then
\[
\frac{\overline{F}^{\ast}\left(  X_{n-k:n}w\right)  }{\overline{F}^{\ast
}\left(  X_{n-k:n}w\right)  }=\left(  1+o_{\mathbf{P}}\left(  1\right)
\right)  w^{-1/\gamma+\epsilon/2},
\]
uniformly on $w>1.$ Therefore the previous quantity reduces into
\[
\left(  1+o_{\mathbf{P}}\left(  1\right)  \right)  \left(  1+\frac
{1/\gamma_{2}+\epsilon/2}{-1/\gamma_{1}+\epsilon}\right)  x^{-1/\gamma
_{1}+\epsilon}.
\]
Thereby the first expression between two brackets in $\left(  \ref{M1-tild}%
\right)  $ equals $O_{\mathbf{P}}\left(  x^{1/\gamma_{1}-\epsilon}\right)  .$
Let us consider the second factor in $\left(  \ref{M1-tild}\right)  .$ By
similar arguments as used for the first factor, we show that%
\[
x^{1/\gamma_{2}+\epsilon/2}\frac{\overline{F}_{n}^{\ast}\left(  X_{n-k:n}%
x\right)  }{\overline{F}^{\ast}\left(  X_{n-k:n}x\right)  }+\left(
1/\gamma_{2}+\epsilon/2\right)  \int_{x}^{\infty}w^{1/\gamma_{2}+\epsilon
/2}\frac{\overline{F}_{n}^{\ast}\left(  X_{n-k:n}x\right)  }{\overline
{F}^{\ast}\left(  X_{n-k:n}x\right)  }dw,
\]
multiplied by $\left(  1+o_{\mathbf{P}}\left(  1\right)  \right)  ,$ uniformly
on $x>1.$ From Lemma $\ref{Lemma1}$\textbf{, }we have%
\[
\frac{\overline{F}_{n}^{\ast}\left(  X_{n-k:n}w\right)  }{\overline{F}^{\ast
}\left(  X_{n-k:n}\right)  }=O_{\mathbf{P}}\left(  w^{-1/\gamma+\epsilon
/2}\right)  ,
\]
which implies that the previous expression equals $O_{\mathbf{P}}\left(
x^{-1/\gamma_{1}+\epsilon}\right)  ,$ thus $\widetilde{\mathbf{M}}_{n1}\left(
x\right)  =O_{\mathbf{P}}\left(  w^{-1/\gamma+\epsilon}\right)  $ and
therefore
\[
\sqrt{k}\mathbf{M}_{n1}\left(  x\right)  =\left(  k/n\right)  ^{1/2-\epsilon
\gamma}O_{\mathbf{P}}\left(  w^{-1/\gamma_{1}+\epsilon}\right)  .
\]
By assumption $k/n\rightarrow0,$ it follows that $\sqrt{k}\mathbf{M}%
_{n1}\left(  x\right)  =o_{\mathbf{P}}\left(  x^{-1/\gamma_{1}+\epsilon
}\right)  $ which meets the result of $\left(  \ref{M1-tild}\right)  .$ Let
now consider the second term $\mathbf{M}_{n2}\left(  x\right)  $ which may be
rewritten into%
\begin{align*}
&  -x^{-1/\gamma_{1}}\frac{k/n}{\overline{F}^{\ast}\left(  X_{n-k:n}\right)
}\frac{\overline{\mathbf{F}}\left(  X_{n-k:n};\theta\right)  }{\overline
{\mathbf{F}}\left(  xX_{n-k:n};\theta\right)  }\frac{\overline{\mathbf{G}%
}_{\theta}\left(  X_{n-k:n}\right)  /\overline{F}^{\ast}\left(  X_{n-k:n}%
\right)  }{\overline{\mathbf{F}}\left(  X_{n-k:n};\theta\right)  }\\
&  \times\int_{x}^{\infty}\frac{\overline{\mathbf{G}}_{\theta}\left(
X_{n-k:n}\right)  }{\overline{\mathbf{G}}_{\theta}\left(  X_{n-k:n}w\right)
}d\frac{\overline{F}_{n}^{\ast}\left(  X_{n-k:n}w\right)  -\overline{F}^{\ast
}\left(  X_{n-k:n}w\right)  }{k/n}.
\end{align*}
In view of Potter's inequalities, it is clear that%
\[
\frac{\overline{\mathbf{F}}\left(  X_{n-k:n};\theta\right)  }{\overline
{F}^{\ast}\left(  X_{n-k:n}\right)  /\overline{\mathbf{G}}_{\theta}\left(
X_{n-k:n}\right)  }\overset{\mathbf{P}}{\rightarrow}\frac{\gamma_{1}}{\gamma
}P\left(  \theta\right)
\]
and%
\[
\frac{\overline{\mathbf{F}}\left(  X_{n-k:n};\theta\right)  }{\overline
{\mathbf{F}}\left(  xX_{n-k:n};\theta\right)  }\overset{\mathbf{P}%
}{\rightarrow}x^{1/\gamma_{1}}.
\]
Note that $\overline{F}^{\ast}\left(  X_{n-k:n}\right)  \overset{d}{=}%
\xi_{k+1:n}$ and Smirnov's lemma (see, e.g., Lemma 2.2.3 in de Haan and
Feriera, 2006) implies that $\frac{n}{k}\xi_{k+1:n}\overset{\mathbf{P}%
}{\rightarrow}1,$ hence $\frac{n}{k}\overline{F}^{\ast}\left(  X_{n-k:n}%
\right)  =1+o_{\mathbf{P}}\left(  1\right)  .$ Therefore
\[
\mathbf{M}_{n2}\left(  x\right)  =-\left(  1+o_{\mathbf{P}}\left(  1\right)
\right)  \frac{\gamma}{\gamma_{1}}\int_{x}^{\infty}\frac{\overline{\mathbf{G}%
}_{\theta}\left(  X_{n-k:n}\right)  }{\overline{\mathbf{G}}_{\theta}\left(
X_{n-k:n}w\right)  }d\frac{\overline{F}_{n}^{\ast}\left(  X_{n-k:n}w\right)
-\overline{F}^{\ast}\left(  X_{n-k:n}w\right)  }{k/n}.
\]
On the other hand, by using an integration by parts yields%
\[
\mathbf{M}_{n2}\left(  x\right)  =\left(  1+o_{\mathbf{P}}\left(  1\right)
\right)  \frac{\gamma_{1}}{\gamma}\left(  \mathbf{M}_{n2}^{\left(  1\right)
}\left(  x\right)  +\mathbf{M}_{n2}^{\left(  2\right)  }\left(  x\right)
\right)  ,
\]
where%
\[
\mathbf{M}_{n2}^{\left(  1\right)  }\left(  x\right)  :=\int_{x}^{\infty}%
\frac{\overline{F}_{n}^{\ast}\left(  X_{n-k:n}w;\theta\right)  -\overline
{F}^{\ast}\left(  X_{n-k:n}w;\theta\right)  }{k/n}d\frac{\overline{\mathbf{G}%
}_{\theta}\left(  X_{n-k:n}\right)  }{\overline{\mathbf{G}}_{\theta}\left(
X_{n-k:n}w\right)  }%
\]
and
\[
\mathbf{M}_{n2}^{\left(  2\right)  }\left(  x\right)  :=\frac{\overline
{\mathbf{G}}_{\theta}\left(  X_{n-k:n}\right)  }{\overline{\mathbf{G}}%
_{\theta}\left(  X_{n-k:n}x\right)  }\frac{\overline{F}_{n}^{\ast}\left(
X_{n-k:n}x;\theta\right)  -\overline{F}^{\ast}\left(  xX_{n-k:n}%
;\theta\right)  }{k/n}.
\]
By using the change of variables $t=\overline{\mathbf{G}}_{\theta}\left(
X_{n-k:n}\right)  /\overline{\mathbf{G}}_{\theta}\left(  X_{n-k:n}w\right)  ,$
it is easy to verify that%
\[
\mathbf{M}_{n2}^{\left(  1\right)  }\left(  x\right)  =\int_{\frac
{\overline{\mathbf{G}}_{\theta}\left(  X_{n-k:n}\right)  }{\overline
{\mathbf{G}}_{\theta}\left(  X_{n-k:n}x\right)  }}^{\infty}\frac{n}{k}\left(
\overline{F}_{n}^{\ast}\left(  \vartheta_{n}\left(  t;\theta\right)  \right)
-\overline{F}^{\ast}\left(  \vartheta_{n}\left(  t;\theta\right)  \right)
\right)  dt,
\]
where $\vartheta_{n}\left(  t;\theta\right)  :=\frac{n}{k}\overline{F}^{\ast
}\left(  \mathbf{G}_{\theta}^{\leftarrow}\left(  1-\overline{\mathbf{G}%
}_{\theta}\left(  X_{n-k:n}\right)  t^{-1}\right)  \right)  .$ Observe that%
\[
\mathbf{M}_{n2}^{\left(  1\right)  }\left(  x\right)  =\int_{\frac
{\overline{\mathbf{G}}_{\theta}\left(  X_{n-k:n}\right)  }{\overline
{\mathbf{G}}_{\theta}\left(  X_{n-k:n}x\right)  }}^{\infty}\left(
\mathbf{U}_{n}\left(  \vartheta_{n}\left(  t;\theta\right)  \right)
-\vartheta_{n}\left(  t;\theta\right)  \right)  dt,
\]
where $\mathbf{U}_{n}$ being the tail empirical df given in $\left(
\ref{tail}\right)  ,$ thereby%
\[
\sqrt{k}\mathbf{M}_{n2}^{\left(  1\right)  }\left(  x\right)  =\int
_{\frac{\overline{\mathbf{G}}_{\theta}\left(  X_{n-k:n}\right)  }%
{\overline{\mathbf{G}}_{\theta}\left(  X_{n-k:n}x\right)  }}^{\infty}%
\alpha_{n}\left(  \vartheta_{n}\left(  t;\theta\right)  \right)  dt,
\]
where $\alpha_{n}$ being the tail empirical process defined in $\left(
\ref{alfan}\right)  .$ Let us decompose the previous integral into%
\begin{align*}
&  \int_{\frac{\overline{\mathbf{G}}_{\theta}\left(  X_{n-k:n}\right)
}{\overline{\mathbf{G}}_{\theta}\left(  X_{n-k:n}x\right)  }}^{\infty}\left(
\alpha_{n}\left(  \vartheta_{n}\left(  t;\theta\right)  \right)  -W\left(
\vartheta_{n}\left(  t;\theta\right)  \right)  \right)  dt+\int_{\frac
{\overline{\mathbf{G}}_{\theta}\left(  X_{n-k:n}\right)  }{\overline
{\mathbf{G}}_{\theta}\left(  X_{n-k:n}x\right)  }}^{\infty}W\left(
\vartheta_{n}\left(  t;\theta\right)  \right)  dt\\
&  =S_{n}+R_{n}.
\end{align*}
By applying weak approximation $\left(  \ref{weak-alfa}\right)  $ we get
\[
S_{n}=o_{\mathbf{P}}\left(  1\right)  \int_{\frac{\overline{\mathbf{G}%
}_{\theta}\left(  X_{n-k:n}\right)  }{\overline{\mathbf{G}}_{\theta}\left(
X_{n-k:n}x\right)  }}^{\infty}\left(  \vartheta_{n}\left(  t;\theta\right)
\right)  ^{1/2-\epsilon}dt.
\]
Observe that $\overline{F}^{\ast}\left(  \mathbf{G}_{\theta}^{\longleftarrow
}\left(  1-\overline{\mathbf{G}}_{\theta}\left(  X_{n-k:n}\right)  \right)
\right)  =\overline{F}^{\ast}\left(  X_{n-k:n}\right)  ,$ thereby%
\[
\vartheta_{n}\left(  t;\theta\right)  =\frac{n}{k}\overline{F}^{\ast}\left(
X_{n-k:n}\right)  \frac{\overline{F}^{\ast}\left(  \mathbf{G}_{\theta
}^{\longleftarrow}\left(  1-\overline{\mathbf{G}}_{\theta}\left(
X_{n-k:n}\right)  t^{-1}\right)  \right)  }{\overline{F}^{\ast}\left(
\mathbf{G}_{\theta}^{\longleftarrow}\left(  1-\overline{\mathbf{G}}_{\theta
}\left(  X_{n-k:n}\right)  \right)  \right)  }.
\]
It is easy to check that $\overline{F}^{\ast}\left(  \mathbf{G}_{\theta
}^{\longleftarrow}\left(  1-\cdot\right)  \right)  \in\mathcal{RV}\left(
\gamma_{2}/\gamma\right)  ,$ then once again by means of Pooter's inequality,
we show that $\vartheta_{n}\left(  t;\theta\right)  =\left(  1+o_{\mathbf{P}%
}\left(  1\right)  \right)  t^{-\gamma_{2}/\gamma+\epsilon},$ therefore%
\[
S_{n}=o_{\mathbf{P}}\left(  1\right)  \int_{\frac{\overline{\mathbf{G}%
}_{\theta}\left(  X_{n-k:n}\right)  }{\overline{\mathbf{G}}_{\theta}\left(
X_{n-k:n}x\right)  }}^{\infty}\left(  t^{-\gamma_{2}/\gamma+\epsilon}\right)
^{1/2-\epsilon}dt.
\]
By using an elementary integration we get
\[
S_{n}=o_{\mathbf{P}}\left(  1\right)  \left(  \frac{\overline{\mathbf{G}%
}_{\theta}\left(  X_{n-k:n}\right)  }{\overline{\mathbf{G}}_{\theta}\left(
X_{n-k:n}x\right)  }\right)  ^{\left(  -\gamma_{2}/\gamma+\epsilon\right)
\left(  1/2-\epsilon\right)  +1}=o_{\mathbf{P}}\left(  x^{\frac{1}{\gamma_{2}%
}-\frac{1}{2\gamma}+\epsilon}\right)  .
\]
By replacing $\gamma$ by its by its expression given in $\left(
\ref{gamma-ratio}\right)  ,$ we end up with
\[
S_{n}=o_{\mathbf{P}}\left(  x^{\frac{1}{2}\left(  \frac{1}{\gamma_{2}}%
-\frac{1}{\gamma_{1}}\right)  +\epsilon}\right)  .
\]
The term $R_{n}$ may be decomposed into%
\[
\int_{\frac{\overline{\mathbf{G}}_{\theta}\left(  X_{n-k:n}\right)
}{\overline{\mathbf{G}}_{\theta}\left(  X_{n-k:n}x\right)  }}^{x^{1/\gamma
_{2}}}W\left(  \vartheta_{n}\left(  t;\theta\right)  \right)  dt+\int
_{x^{1/\gamma_{2}}}^{\infty}W\left(  \vartheta_{n}\left(  t;\theta\right)
\right)  dt=R_{n1}+R_{n2}.
\]
It is clear that%
\[
\left\vert R_{n1}\right\vert <\left\{  \sup_{t>\frac{\overline{\mathbf{G}%
}_{\theta}\left(  X_{n-k:n}\right)  }{\overline{\mathbf{G}}_{\theta}\left(
X_{n-k:n}x\right)  }}\frac{\left\vert W\left(  \vartheta_{n}\left(
t;\theta\right)  \right)  \right\vert }{\left(  \vartheta_{n}\left(
t;\theta\right)  \right)  ^{\epsilon}}\right\}  \int_{\frac{\overline
{\mathbf{G}}_{\theta}\left(  X_{n-k:n}\right)  }{\overline{\mathbf{G}}%
_{\theta}\left(  X_{n-k:n}x\right)  }}^{x^{1/\gamma_{2}}}\left(  \vartheta
_{n}\left(  t;\theta\right)  \right)  ^{\epsilon}dt.
\]
It is ready to check, by using the change of variables $\vartheta_{n}\left(
t;\theta\right)  =s,$ that the previous first factor between the curly
brackets equals%
\[
\sup_{0<s<\frac{n}{k}\overline{F}^{\ast}\left(  X_{n-k:n}x;\theta\right)
}\frac{\left\vert W\left(  s\right)  \right\vert }{s^{\epsilon}}%
<\sup_{0<s<\frac{n}{k}\overline{F}^{\ast}\left(  X_{n-k:n};\theta\right)
}\frac{\left\vert W\left(  s\right)  \right\vert }{s^{\epsilon}}.
\]
From Lemma 3.2 in \cite{Einmahl2006}\ $\sup_{0<s\leq1}s^{-\delta}\left\vert
W\left(  s\right)  \right\vert =O_{\mathbf{P}}\left(  1\right)  ,$ for any
$0<\delta<1/2,$ then since $n\overline{F}^{\ast}\left(  X_{n-k:n}%
;\theta\right)  /k\overset{\mathbf{P}}{\rightarrow}1,$ as $n\rightarrow
\infty,$ we infer that
\[
\sup_{0<s<\frac{n}{k}\overline{F}^{\ast}\left(  X_{n-k:n};\theta\right)
}s^{-\epsilon}\left\vert W\left(  s\right)  \right\vert =O_{\mathbf{P}}\left(
1\right)  .
\]
for all large $n.$ On the other hand, we already pointed out above that
\[
\vartheta_{n}\left(  t;\theta\right)  =\left(  1+o_{\mathbf{P}}\left(
1\right)  \right)  t^{-\gamma_{2}/\gamma+\epsilon},
\]
which implies that the second factor is equal to%
\[
O_{\mathbf{P}}\left(  1\right)  \int_{\frac{\overline{\mathbf{G}}_{\theta
}\left(  X_{n-k:n}\right)  }{\overline{\mathbf{G}}_{\theta}\left(
X_{n-k:n}x\right)  }}^{x^{1/\gamma_{2}}}\left(  t^{-\gamma_{2}/\gamma
+\epsilon}\right)  ^{\epsilon}dt=O_{\mathbf{P}}\left(  1\right)  \int
_{\frac{\overline{\mathbf{G}}_{\theta}\left(  X_{n-k:n}\right)  }%
{\overline{\mathbf{G}}_{\theta}\left(  X_{n-k:n}x\right)  }}^{x^{1/\gamma_{2}%
}}t^{-\epsilon\gamma_{2}/\gamma+\epsilon}dt,
\]
which after integration yields%
\[
O_{\mathbf{P}}\left(  1\right)  \left\{  \left(  \frac{\overline{\mathbf{G}%
}_{\theta}\left(  X_{n-k:n}\right)  }{\overline{\mathbf{G}}_{\theta}\left(
X_{n-k:n}x\right)  }\right)  ^{-\epsilon\gamma_{2}/\gamma+\epsilon+1}-\left(
x^{-1/\gamma}\right)  ^{-\epsilon\gamma_{2}/\gamma+\epsilon+1}\right\}  .
\]
Recall that from formula $\left(  \ref{gamma-ratio}\right)  $ we have
$\gamma_{2}/\gamma>1,$ then by using the mean value theorem and Pooter's
inequalities, we get $R_{n1}=o_{\mathbf{P}}\left(  x^{-\epsilon}\right)  .$
The second term $R_{n2}$ may be decomposed into%
\[
R_{n2}=\int_{x^{1/\gamma_{2}}}^{\infty}\left(  W\left(  \vartheta_{n}\left(
t;\theta\right)  \right)  -W\left(  t^{-\gamma_{2}/\gamma}\right)  \right)
dt+\int_{x^{1/\gamma_{2}}}^{\infty}W\left(  t^{-\gamma_{2}/\gamma}\right)
dt.
\]
From Proposition B.1.10 in \cite{deHF06}, with high probability,%
\begin{equation}
c_{n}\left(  t;\theta\right)  :=\left\vert \vartheta_{n}\left(  t;\theta
\right)  -t^{-\gamma_{2}/\gamma}\right\vert \leq\epsilon t^{-\gamma_{2}%
/\gamma-\epsilon},\text{ as }n\rightarrow\infty, \label{vn}%
\end{equation}
this means that $\sup_{x>1}\sup_{t>x^{1/\gamma_{2}}}c_{n}\left(
t;\theta\right)  \overset{\mathbf{P}}{\rightarrow}0,$ as $n\rightarrow\infty.$
This implies by using Levy's modulus of continuity of the Wiener process
\citep[see, e.g., Theorem 1.1.1 in ][]{CR81}, that%
\[
\left\vert W\left(  \vartheta_{n}\left(  t;\theta\right)  \right)  -W\left(
t^{-\gamma_{2}/\gamma}\right)  \right\vert \leq2\sqrt{c_{n}\left(
t;\theta\right)  \log\left(  1/c_{n}\left(  t;\theta\right)  \right)  },
\]
with high probability. By using the fact that $\log s<\epsilon s^{-\epsilon},$
for $s\downarrow0$ together with inequality $\left(  \ref{vn}\right)  ,$ we
show that
\[
\left\vert W\left(  \vartheta_{n}\left(  t;\theta\right)  \right)  -W\left(
t^{-\gamma_{2}/\gamma}\right)  \right\vert <2\epsilon t^{-\left(  \gamma
_{2}/\gamma-\epsilon\right)  /2},
\]
uniformly on $t>x^{1/\gamma_{2}},$ it follows that%
\[
\left\vert \int_{x^{1/\gamma_{2}}}^{\infty}\left(  W\left(  \vartheta
_{n}\left(  t;\theta\right)  \right)  -W\left(  t^{-\gamma_{2}/\gamma}\right)
\right)  dt\right\vert =o_{\mathbf{P}}\left(  1\right)  \left\vert
\int_{x^{1/\gamma_{2}}}^{\infty}t^{-\left(  \gamma_{2}/\gamma-\epsilon\right)
/2}dt\right\vert .
\]
Recall that the assumption $\gamma_{1}<\gamma_{2}$ together with the equation
$1/\gamma=1/\gamma_{1}+1/\gamma_{2},$ imply that $\gamma_{2}/\left(
2\gamma\right)  >1,$ thus $-\left(  \gamma_{2}/\gamma-\epsilon\right)
/2+1<0,$ therefore $\left\vert \int_{x^{1/\gamma_{2}}}^{\infty}t^{-\left(
\gamma_{2}/\gamma-\epsilon\right)  /2}dt\right\vert =o_{\mathbf{P}}\left(
x^{-1/\gamma_{1}-\epsilon}\right)  .$ Then we showed that
\[
R_{n1}=o_{\mathbf{P}}\left(  x^{-\epsilon}\right)  \text{ and }R_{n2}%
=\int_{x^{1/\gamma_{2}}}^{\infty}W\left(  t^{-\gamma_{2}/\gamma}\right)
dt+o_{\mathbf{P}}\left(  x^{-1/\gamma_{1}-\epsilon}\right)  ,
\]
hence
\[
\sqrt{k}\mathbf{M}_{n2}^{\left(  1\right)  }\left(  x\right)  =R_{n}%
+S_{n}=\int_{x^{1/\gamma_{2}}}^{\infty}W\left(  t^{-\gamma_{2}/\gamma}\right)
dt+o_{\mathbf{P}}\left(  x^{-1/\gamma_{1}-\epsilon}\right)  +o_{\mathbf{P}%
}\left(  x^{\frac{1}{2}\left(  \frac{1}{\gamma_{2}}-\frac{1}{\gamma_{1}%
}\right)  +\epsilon}\right)  .
\]
It is clear that%
\[
\left(  -\frac{1}{\gamma_{1}}-\epsilon\right)  -\left(  \frac{1}{2}\left(
\frac{1}{\gamma_{2}}-\frac{1}{\gamma_{1}}\right)  +\epsilon\right)
=-\frac{\gamma_{1}+\gamma_{2}+4\epsilon\gamma_{1}\gamma_{2}}{2\gamma_{1}%
\gamma_{2}}<0.
\]
then%
\[
\sqrt{k}\mathbf{M}_{n2}^{\left(  1\right)  }=\int_{x^{1/\gamma_{2}}}^{\infty
}W\left(  t^{-\gamma_{2}/\gamma}\right)  dt+o_{\mathbf{P}}\left(  x^{\frac
{1}{2}\left(  \frac{1}{\gamma_{2}}-\frac{1}{\gamma_{1}}\right)  +\epsilon
}\right)  .
\]
By using similar arguments we end up with%
\[
\sqrt{k}\mathbf{M}_{n2}^{\left(  2\right)  }\left(  x\right)  =x^{1/\gamma
_{2}}W\left(  t^{-1/\gamma}\right)  +o_{\mathbf{P}}\left(  x^{-\frac{1}%
{\gamma_{1}}+\epsilon}\right)  ,
\]
therefore we omit further details.%
\[
\sqrt{k}\mathbf{M}_{n2}=\frac{\gamma}{\gamma_{1}}x^{1/\gamma_{2}}W\left(
t^{-1/\gamma}\right)  +\frac{\gamma}{\gamma_{1}}\int_{x^{1/\gamma_{2}}%
}^{\infty}W\left(  t^{-\gamma_{2}/\gamma}\right)  dt+o_{\mathbf{P}}\left(
x^{\frac{1}{2}\left(  \frac{1}{\gamma_{2}}-\frac{1}{\gamma_{1}}\right)
+\epsilon}\right)  .
\]
Let us now focus of the term $\mathbf{M}_{n3}.$ By letting $x=1$ in the
previous weak approximation we infer that
\begin{equation}
\sqrt{k}\frac{\overline{\mathbf{F}}_{n}\left(  X_{n-k:n};\theta\right)
-\overline{\mathbf{F}}\left(  X_{n-k:n};\theta\right)  }{\overline{\mathbf{F}%
}\left(  X_{n-k:n};\theta\right)  }=\frac{\gamma}{\gamma_{1}}W\left(
1\right)  +\frac{\gamma}{\gamma_{1}}\int_{1}^{\infty}W\left(  t^{-\gamma
_{2}/\gamma}\right)  dt+o_{\mathbf{P}}\left(  1\right)  . \label{w1}%
\end{equation}
which implies that%
\[
\sqrt{k}\frac{\overline{\mathbf{F}}_{n}\left(  X_{n-k:n};\theta\right)
-\overline{\mathbf{F}}\left(  X_{n-k:n};\theta\right)  }{\overline{\mathbf{F}%
}\left(  X_{n-k:n};\theta\right)  }=O_{\mathbf{P}}\left(  1\right)  .
\]
In other terms, we have%
\begin{equation}
\frac{\overline{\mathbf{F}}_{n}\left(  X_{n-k:n};\theta\right)  }%
{\overline{\mathbf{F}}\left(  X_{n-k:n};\theta\right)  }=1+O_{\mathbf{P}%
}\left(  k^{-1/2}\right)  . \label{app}%
\end{equation}
The regular variation of $\overline{\mathbf{F}}\left(  \cdot;\theta\right)  $
and $\left(  \ref{app}\right)  $ together imply that%
\begin{equation}
\frac{\overline{\mathbf{F}}\left(  xX_{n-k:n};\theta\right)  }{\overline
{\mathbf{F}}_{n}\left(  X_{n-k:n};\theta\right)  }=x^{-1/\gamma_{1}%
}+o_{\mathbf{P}}\left(  x^{-1/\gamma_{1}+\epsilon}\right)  . \label{app1}%
\end{equation}
By combining the results $\left(  \ref{w1}\right)  $ and $\left(
\ref{app1}\right)  $ we get
\[
\sqrt{k}\mathbf{M}_{n3}\left(  x\right)  =-x^{-1/\gamma_{2}}\left(
\frac{\gamma}{\gamma_{1}}W\left(  1\right)  +\frac{\gamma}{\gamma_{1}}\int
_{1}^{\infty}W\left(  t^{-\gamma_{2}/\gamma}\right)  dt\right)  +o_{\mathbf{P}%
}\left(  x^{-1/\gamma_{1}+\epsilon}\right)  .
\]
For the forth term we write
\[
\sqrt{k}\mathbf{M}_{n4}\left(  x\right)  =\left(  \frac{\overline{\mathbf{F}%
}\left(  xX_{n-k:n};\theta\right)  }{\overline{\mathbf{F}}_{n}\left(
X_{n-k:n};\theta\right)  }-x^{-1/\gamma_{1}}\right)  \left(  \sqrt{k}%
\frac{\overline{\mathbf{F}}_{n}\left(  xX_{n-k:n};\theta\right)
-\overline{\mathbf{F}}\left(  xX_{n-k:n};\theta\right)  }{\overline
{\mathbf{F}}\left(  xX_{n-k:n};\theta\right)  }\right)  .
\]
From $\left(  \ref{app1}\right)  $ the first factor of the previous equation
equals\textbf{\ }$o_{\mathbf{P}}\left(  x^{-1/\gamma_{1}+\epsilon}\right)  .$
On the other hand by using the change of variables $s=t^{-\gamma_{2}/\gamma},$
yields%
\[
\int_{x^{1/\gamma_{2}}}^{\infty}W\left(  t^{-\gamma_{2}/\gamma}\right)
dt=\frac{\gamma}{\gamma_{2}}\int_{0}^{x^{-1/\gamma}}s^{-\gamma/\gamma_{2}%
-1}W\left(  s\right)  ds.
\]
Since $\sup_{0<s<1}s^{-1/2+\epsilon}\left\vert W\left(  s\right)  \right\vert
=O_{\mathbf{P}}\left(  1\right)  ,$ then we easily show that
\[
\int_{x^{1/\gamma_{2}}}^{\infty}W\left(  t^{-\gamma_{2}/\gamma}\right)
dt=O_{\mathbf{P}}\left(  x^{\frac{1}{2}\left(  \frac{1}{\gamma_{2}}-\frac
{1}{\gamma_{1}}\right)  +\epsilon}\right)  ,
\]
it follows that $\sqrt{k}\mathbf{M}_{n2}=O_{\mathbf{P}}\left(  x^{\frac{1}%
{2}\left(  \frac{1}{\gamma_{2}}-\frac{1}{\gamma_{1}}\right)  +\epsilon
}\right)  $ as well. Therefore
\[
\sqrt{k}\frac{\overline{\mathbf{F}}_{n}\left(  xX_{n-k:n};\theta\right)
-\overline{\mathbf{F}}\left(  xX_{n-k:n};\theta\right)  }{\overline
{\mathbf{F}}\left(  xX_{n-k:n};\theta\right)  }=x^{1/\gamma_{1}}O_{\mathbf{P}%
}\left(  x^{\frac{1}{2}\left(  \frac{1}{\gamma_{2}}-\frac{1}{\gamma_{1}%
}\right)  +\epsilon}\right)  =O_{\mathbf{P}}\left(  x^{\frac{1}{2\gamma
}+\epsilon}\right)  .
\]
Hence we have
\[
\sqrt{k}\mathbf{M}_{n4}\left(  x\right)  =o_{\mathbf{P}}\left(  x^{-1/\gamma
_{1}+\epsilon}\right)  O_{\mathbf{P}}\left(  x^{\frac{1}{2\gamma}+\epsilon
}\right)  =o_{\mathbf{P}}\left(  x^{\frac{1}{2}\left(  \frac{1}{\gamma_{2}%
}-\frac{1}{\gamma_{1}}\right)  +\epsilon}\right)  .
\]
By assumption $\overline{\mathbf{F}}$ satisfies the second order condition of
regular variation function $\left(  \ref{second-order}\right)  ,$ this means
that for%
\begin{equation}
\lim_{t\rightarrow\infty}\dfrac{\overline{\mathbf{F}}\left(  tx\right)
/\overline{\mathbf{F}}\left(  t\right)  -x^{-1/\gamma_{1}}}{\mathbf{A}\left(
t\right)  }=x^{-1/\gamma_{1}}\dfrac{x^{\rho_{1}/\gamma_{1}}-1}{\rho_{1}%
\gamma_{1}}, \label{seconF}%
\end{equation}
for any $x>0,$ where $\rho_{1}<0$ is the second-order parameter and
$\mathbf{A}$ is $\mathcal{RV}\left(  \rho_{1}/\gamma_{1}\right)  .$ The
uniform inequality corresponding to $\left(  \ref{seconF}\right)  $ says:
there exist $t_{0}>0,$ such that for any $t>t_{0},$ we have
\[
\left\vert \dfrac{\overline{\mathbf{F}}\left(  tx\right)  /\overline
{\mathbf{F}}\left(  t\right)  -x^{-1/\gamma_{1}}}{\mathbf{A}\left(  t\right)
}-x^{-1/\gamma_{1}}\dfrac{x^{\rho_{1}/\gamma_{1}}-1}{\rho_{1}\gamma_{1}%
}\right\vert <\epsilon x^{-1/\gamma_{1}+\rho_{1}/\gamma_{1}+\epsilon},
\]
see for instance assertion (2.3.23) of Theorem 2.3.9 in \cite{deHF06}. \ It is
easy to check that the later inequality implies that
\begin{align*}
\sqrt{k}\mathbf{M}_{n5}\left(  x\right)   &  =\sqrt{k}\left(  \frac
{\overline{\mathbf{F}}\left(  xX_{n-k:n};\theta\right)  }{\overline
{\mathbf{F}}\left(  X_{n-k:n};\theta\right)  }-x^{-1/\gamma_{1}}\right) \\
&  =x^{-1/\gamma_{1}}\dfrac{x^{\rho_{1}/\gamma_{1}}-1}{\rho_{1}\gamma_{1}%
}\sqrt{k}\mathbf{A}\left(  X_{n-k:n}\right)  +o_{\mathbf{P}}\left(
x^{-1/\gamma_{1}}\dfrac{x^{\rho_{1}/\gamma_{1}}-1}{\rho_{1}\gamma_{1}}\sqrt
{k}\mathbf{A}\left(  X_{n-k:n}\right)  \right)  .
\end{align*}
Recall that $a_{k}:=F^{\ast\leftarrow}\left(  1-k/n\right)  $ and notice that
$X_{n-k:n}/a_{k}\overset{\mathbf{P}}{\rightarrow}1$ as $n\rightarrow\infty,$
then in view of the regular variation of $\mathbf{A}$ we infer that
$\mathbf{A}\left(  X_{n-k:n}\right)  =\left(  1+o_{\mathbf{P}}\left(
1\right)  \right)  \mathbf{A}\left(  a_{k}\right)  .$ On the other hand, by
assumption $\sqrt{k}\mathbf{A}\left(  a_{k}\right)  $ is asymptotically
bounded, therefore%
\[
\sqrt{k}\mathbf{M}_{n5}\left(  x\right)  =x^{-1/\gamma_{1}}\dfrac{x^{\rho
_{1}/\gamma_{1}}-1}{\rho_{1}\gamma_{1}}\sqrt{k}\mathbf{A}\left(  a_{k}\right)
+o_{\mathbf{P}}\left(  x^{-1/\gamma_{1}}\right)  .
\]
To summarize, as this stage we showed that%
\begin{align*}
\mathbf{D}_{n}\left(  x;\widehat{\theta}\right)   &  =\frac{\gamma}{\gamma
_{1}}x^{1/\gamma_{2}}W\left(  t^{-1/\gamma}\right)  +\frac{\gamma}{\gamma_{1}%
}\int_{x^{1/\gamma_{2}}}^{\infty}W\left(  t^{-\gamma_{2}/\gamma}\right)  dt\\
&  \ \ \ \ \ -x^{-1/\gamma_{2}}\left(  \frac{\gamma}{\gamma_{1}}W\left(
1\right)  +\frac{\gamma}{\gamma_{1}}\int_{1}^{\infty}W\left(  t^{-\gamma
_{2}/\gamma}\right)  dt\right) \\
&  \ \ \ \ \ \ \ \ \ \ \ \ \ \ \ \ +x^{-1/\gamma_{1}}\dfrac{x^{\rho_{1}%
/\gamma_{1}}-1}{\rho_{1}\gamma_{1}}\sqrt{k}\mathbf{A}\left(  a_{k}\right)
+\mathcal{\varsigma}\left(  x\right)  ,
\end{align*}
where $\mathcal{\varsigma}\left(  x\right)  :=o_{\mathbf{P}}\left(
x^{-1/\gamma_{1}+\epsilon}\right)  +o_{\mathbf{P}}\left(  x^{-1/\gamma_{1}%
}\right)  +o_{\mathbf{P}}\left(  x^{\frac{1}{2}\left(  \frac{1}{\gamma_{2}%
}-\frac{1}{\gamma_{1}}\right)  +\epsilon}\right)  .$ By using a change of
variables, we show that sum of the first three terms equals the Gaussian
process $\Gamma\left(  x;W\right)  $ stated in Theorem $\ref{Theorem1}$.
Recall that $\gamma_{1}<\gamma_{2}$ and
\[
\frac{1}{2}\left(  \frac{1}{\gamma_{2}}-\frac{1}{\gamma_{1}}\right)
+\epsilon<0,
\]
then it is easy to verify that $\mathcal{\varsigma}\left(  x\right)
=o_{\mathbf{P}}\left(  x^{\frac{1}{2}\left(  \frac{1}{\gamma_{2}}-\frac
{1}{\gamma_{1}}\right)  +\epsilon}\right)  .$ It follows that%
\begin{align*}
&  x^{\epsilon}\left\{  \mathbf{D}_{n}\left(  x;\widehat{\theta}\right)
-\Gamma\left(  x;W\right)  -x^{-1/\gamma_{1}}\dfrac{x^{\rho_{1}/\gamma_{1}}%
-1}{\rho_{1}\gamma_{1}}\sqrt{k}\mathbf{A}\left(  a_{k}\right)  \right\} \\
&  =o_{\mathbf{P}}\left(  x^{\frac{1}{2}\left(  \frac{1}{\gamma_{2}}-\frac
{1}{\gamma_{1}}\right)  +2\epsilon}\right)  =o_{\mathbf{P}}\left(  1\right)  ,
\end{align*}
uniformly on $x>1,$ therefore%
\[
\sup_{x>1}x^{\epsilon}\left\vert \mathbf{D}_{n}\left(  x;\widehat{\theta
}\right)  -\Gamma\left(  x;W\right)  -x^{-1/\gamma_{1}}\dfrac{x^{\rho
_{1}/\gamma_{1}}-1}{\rho_{1}\gamma_{1}}\sqrt{k}\mathbf{A}\left(  a_{k}\right)
\right\vert =o_{\mathbf{P}}\left(  1\right)  ,
\]
for any samll $0<\epsilon<1/2,$ which completes the proof of Theorem
$\ref{Theorem1}$.

\subsection{Proof of Theorem \ref{Theorem2}.}

From the representation $\left(  \ref{rep}\right)  $ we write%
\[
\widehat{\gamma}_{1}-\gamma_{1}=T_{n1}+T_{n2}+T_{n3},
\]
where%
\[
T_{n1}:=k^{-1/2}\int_{1}^{\infty}x^{-1}\left\{  \mathbf{D}_{n}\left(
x;\widehat{\theta};\gamma_{1}\right)  -\Gamma\left(  x;W\right)
-x^{-1/\gamma_{1}}\dfrac{x^{\rho_{1}/\gamma_{1}}-1}{\rho_{1}\gamma_{1}}%
\sqrt{k}\mathbf{A}\left(  a_{k}\right)  \right\}  dx
\]%
\[
T_{n2}:=k^{-1/2}\int_{1}^{\infty}x^{-1}\Gamma\left(  x;W\right)  dx
\]
and
\[
T_{n3}:=-\mathbf{A}\left(  a_{k}\right)  \int_{1}^{\infty}x^{-1/\gamma_{1}%
-1}\dfrac{x^{\rho_{1}/\gamma_{1}}-1}{\rho_{1}\gamma_{1}}dx.
\]
By using $\ref{Theorem1}$ yields $T_{n1}=o_{P}\left(  k^{-1/2}\right)
\int_{1}^{\infty}x^{-1+\epsilon}dx=o_{P}\left(  k^{-1/2}\right)  $.
Since\textbf{ }$\mathbf{E}\left\vert W\left(  s\right)  \right\vert \leq
s^{1/2},$ then it is easy to show that $\int_{1}^{\infty}x^{-1}\Gamma\left(
x;W\right)  dx=O_{\mathbf{P}}\left(  1\right)  ,$ it follows that
$T_{n2}=O_{\mathbf{P}}\left(  k^{-1/2}\right)  .$ By using an elementary
integration, we get $T_{n3}:=\frac{\mathbf{A}\left(  a_{k}\right)  }%
{1-\rho_{1}}.$ Since both $k^{-1/2}$ and $\mathbf{A}\left(  a_{k}\right)  $
tend to zero as $n\rightarrow\infty,$ then $T_{n3}=o\left(  1\right)  ,$ it
follows that $\widehat{\gamma}_{1}\overset{\mathbf{P}}{\rightarrow}\gamma
_{1},$ which gives the first result of Theorem. To establish the asymptotic
normality%
\[
\sqrt{k}\left(  \widehat{\gamma}_{1}-\gamma_{1}\right)  =\sqrt{k}T_{n1}%
+\sqrt{k}T_{n2}+\sqrt{k}T_{n3},
\]
where
\[
\sqrt{k}T_{n1}=o_{\mathbf{P}}\left(  1\right)  ,\sqrt{k}T_{n2}=\int
_{1}^{\infty}x^{-1}\Gamma\left(  x;W\right)  dx
\]
and
\[
\sqrt{k}T_{n2}=\frac{\sqrt{k}\mathbf{A}\left(  a_{k}\right)  }{1-\rho_{1}}.
\]
Note that $\Gamma\left(  x;W\right)  $ is a centred Gaussian process and by
using the assumption $\sqrt{k}\mathbf{A}\left(  a_{k}\right)  \rightarrow
\lambda<\infty,$ we end up with%
\[
\sqrt{k}\left(  \widehat{\gamma}_{1}-\gamma_{1}\right)  \overset{\mathcal{D}%
}{\rightarrow}\mathcal{N}\left(  \frac{\lambda}{1-\rho_{1}},\mathbf{E}\left[
\int_{1}^{\infty}x^{-1}\Gamma\left(  x;W\right)  dx\right]  ^{2}\right)  .
\]
By using elementary calculation we show that $\mathbf{E}\left[  \int
_{1}^{\infty}x^{-1}\Gamma\left(  x;W\right)  dx\right]  ^{2}=\sigma^{2},$ that
we omit the details.

\section{\textbf{Conclusion}}

\noindent In basis on a semiparametric estimator of the underlying
distribution function, we proposed a new estimation method to the tail index
of Pareto-type distributions for randomly right-truncated data. Compared with
the existing ones, this estimator behaves well both in terms of bias and rmse.
A useful weak approximation of the corresponding tail empirical process
allowed us to establish both the consistency and asymptotic normality of the
proposed estimator.

\section{\textbf{Appendix}}

\begin{lemma}
\label{Lemma1}For any small $\epsilon>0,$ we have%
\[
\frac{\overline{F}_{n}^{\ast}\left(  X_{n-k:n}w\right)  }{\overline{F}^{\ast
}\left(  X_{n-k:n}\right)  }=O_{\mathbf{P}}\left(  w^{-1/\gamma+\epsilon
/2}\right)  ,\text{ uniformly on }w\geq1.
\]

\end{lemma}

\begin{proof}
Let $V_{n}\left(  t\right)  :=n^{-1}\sum_{i=1}^{n}\mathbf{1}\left(  \xi
_{i}\leq t\right)  $ be the uniform empirical df pertaining to the sample
$\xi_{i}:=\overline{F}^{\ast}\left(  X_{i}\right)  ,$ $i=1,...,n,$ of iid
uniform$\left(  0,1\right)  $ rv's. It is clear that, for an arbitrary $x,$ we
have $V_{n}\left(  \overline{F}^{\ast}\left(  x\right)  \right)  =\overline
{F}_{n}^{\ast}\left(  x\right)  $ almost surely. From Assertion 7 in
\cite{SW86} (page 415), $V_{n}\left(  t\right)  /t=O_{\mathbf{P}}\left(
1\right)  $ uniformly on $1/n\leq t\leq1,$ this implies that
\begin{equation}
\frac{\overline{F}_{n}^{\ast}\left(  X_{n-k:n}w\right)  }{\overline{F}^{\ast
}\left(  X_{n-k:n}w\right)  }=O_{\mathbf{P}}\left(  1\right)  ,\text{
uniformly on }w\geq1. \label{a}%
\end{equation}
On the other hand, by applying Potter's inequalities $\left(  \ref{pooter}%
\right)  $ to $\overline{F}^{\ast},$ we get
\begin{equation}
\frac{\overline{F}^{\ast}\left(  X_{n-k:n}w\right)  }{\overline{F}^{\ast
}\left(  X_{n-k:n}\right)  }=O_{\mathbf{P}}\left(  w^{-1/\gamma+\epsilon
/2}\right)  ,\text{ uniformly on }w\geq1. \label{b}%
\end{equation}
Combining the two statements $\left(  \ref{a}\right)  $ and $\left(
\ref{b}\right)  $ gives the desired results.
\end{proof}

\end{document}